\documentclass[11pt]{amsart}
\usepackage{geometry}                
\usepackage{graphicx}
\usepackage{amssymb}
\usepackage{epstopdf}
\newtheorem{theorem}{Theorem}[section]
\newtheorem{lemma}{Lemma}[section]
\newtheorem{corollary}{Corollary}[section]
\newtheorem{proposition}{Proposition}[section]

\newtheorem{definition}{Definition}[section]
\newtheorem{example}{Example}[section]
\newtheorem{remark}{Remark}[section]

\numberwithin{equation}{section}

\def\Z{\mathbb Z}

\def\R{\mathbb R}
\def\P{\mathbb P}
\def\C{\mathbb C}

\def\D{\Delta}
\def\Om{\Omega}

\def\om{\omega}
\def\d{\partial}
\def\s{\sigma}
\def\e{\epsilon}
\def\a{\alpha}
\def\b{\beta}
\def\g{\gamma}

\def\sI{{\mathsf I}}
\def\sM{{\mathsf M}}
\def\sR{{\mathsf R}}

\def\bfc{{\mathbf c}}
\def\bfc\Theta{\mathcal c\Theta}
\def\cP{\mathcal P}
\def\bfOm{{\mathbf \Omega}}
\def\paOm#1#2{{\mathbf \Omega}_{#1,\,|\sim|' #2}}

\title[Spaces of polynomials as Grassmanians for traversing flows]{Spaces of polynomials with constrained divisors as Grassmanians for traversing flows}
\author{Gabriel Katz}
\date{03/22/2022}                                           

\address{MIT, Department of Mathematics, 77 Massachusetts Ave., Cambridge, MA 02139, U.S.A.}

\email {gabkatz@gmail.com }

\begin{document}
\maketitle 

\begin{abstract} We study {\sf traversing} vector flows $v$ on smooth compact manifolds $X$ with boundary. For a given compact manifold $\hat X$, equipped with a traversing vector field $\hat v$ which is {\sf convex} with respect to $\partial\hat X$, we consider submersions/embeddings $\alpha: X \to \hat X$ such that $\dim X = \dim \hat X$ and $\alpha(\partial X)$ avoids a priory chosen tangency patterns $\Theta$ to the $\hat v$-trajectories. In particular, for each $\hat v$-trajectory $\hat\gamma$, we restrict the cardinality of $\hat\gamma \cap \alpha(\partial X)$ by an even number $d$. We call $(\hat X, \hat v)$ a {\sf convex pseudo-envelop/envelop} of the pair $(X, v)$.  Here the vector field $v = \alpha^\dagger(\hat v)$ is the $\alpha$-transfer of $\hat v$ to $X$. 

For a fixed $(\hat X, \hat v)$, we introduce an equivalence relation among convex pseudo-envelops/ envelops $\alpha: (X, v) \to (\hat X, \hat v)$, which we call a {\sf quasitopy}. The notion of quasitopy is a crossover between bordisms of pseudo-envelops and their pseudo-isotopies. In the study of quasitopies $\mathcal{QT}_d(Y, \mathbf c\Theta)$, the spaces $\mathcal P_d^{\mathbf c\Theta}$ of real univariate polynomials of degree $d$ with real divisors whose combinatorial types avoid the closed poset $\Theta$ play the classical role of Grassmanians. 

We compute, in the homotopy-theoretical terms that involve $(\hat X, \hat v)$ and $\mathcal P_d^{\mathbf c\Theta}$, the quasitopies of convex envelops which avoid the $\Theta$-tangency patterns. We introduce characteristic classes of pseudo-envelops and show that they are invariants of their quasitopy classes. Then we prove that the quasitopies $\mathcal{QT}_d(Y, \mathbf c\Theta)$ often stabilize, as $d \to \infty$.
\end{abstract}

\section{Introduction}

This paper is the second in a series, which is inspired by the works of Arnold \cite{Ar}, \cite{Ar1},  and Vassiliev \cite{V}. It is a direct continuation of \cite{K9}, where many ideas of this article are present. Here we apply these ideas to traversing flows on manifolds with boundary. As \cite{K9}, this article relies heavily on computations from \cite{KSW1}, \cite{KSW2}. While \cite{K9} is studying immersions of compact $n$-dimensional manifolds into products $\R \times Y$, where $Y$ is a compact $n$-dimensional manifold, the present paper deals with special traversing flows on compact $(n+1)$-dimensional manifolds $X$, the flows that admit a ``virtual global section". The existence of such sections allows to establish a transparent correspondence between the universe of immersions into the products and the universe of traversing flows with virtual sections. In fact, with this correspondence in place, many results of the present article are just reformulations of the corresponding results from \cite{K9}, thus motivating and justifying them. However, we prove also many propositions  with no analogues in \cite{K9} (for example, Proposition \ref{prop.homology_sphere}, Theorem \ref{th.traversally_generic}, Corollary \ref{cor.codim}, Proposition \ref{prop.choice_of_convex_field}, and Theorem \ref{th.counting_trajectories}).\smallskip

Let us describe here our results informally, in a manner that clarifies their nature, but does not involve their most general forms, which carry the burden of combinatorial decorations. \smallskip

As we have mentioned above, we study {\sf traversing} vector flows $v$ (see Definition \ref{def.traversing} and \cite{K1}, \cite{K2}) on smooth compact $(n+1)$-dimensional manifolds $X$ with boundary. Every traversing flow admits a Lyapunov function, and this fact may serve as a working definition of such flows. We are interested in the {\sf combinatorial patterns} $\om$ that describe how the $v$-trajectories are tangent to the boundary $\d X$. Specifically, we are concerned with the traversing vector flows whose tangency patterns $\om$ do not belong to a given closed poset $\Theta$. Depending on $\Theta$, the very existence of such flows puts strong restrictions on the topology of $X$ (see \cite{K1}-\cite{K8}, \cite{K10}). For example, in Fig.\ref{fig.A1}, $\Theta$ includes the cubic and the double tangencies of $\d X$ to the family of $u$-directed lines. In other words, in this figure, the tangency patterns $\om = (3)$ and $\om = (22)$ are forbidden. \smallskip

In order to simplify the investigation of such flows, we ``envelop" them into, so called, {\sf convex pseudo-envelops} $(\hat X, \hat v)$. A convex pseudo-envelop consists of: {\bf(i)} a compact smooth $(n+1)$-manifold $\hat X$, {\bf(ii)} a traversing vector field $\hat v$ on it such that the boundary $\d \hat X$ is {\sf convex} (see Definition \ref{def.convex}) with respect to the $\hat v$-flow, {\bf(iii)} a submersion $\a: X \to \hat X$ such that $\a(\d X)$ avoids a forbidden tangency patterns $\Theta$ to the $\hat v$-trajectories, and {\bf(iv)} the pull-back $\a^\dagger(\hat v)$ of $\hat v$, coinciding with the given vector field $v$ (see Fig.\ref{fig.A1}, in which $\hat X$ is a cylinder). 

\begin{figure}[ht]
\centerline{\includegraphics[height=2.6in,width=3.6in]{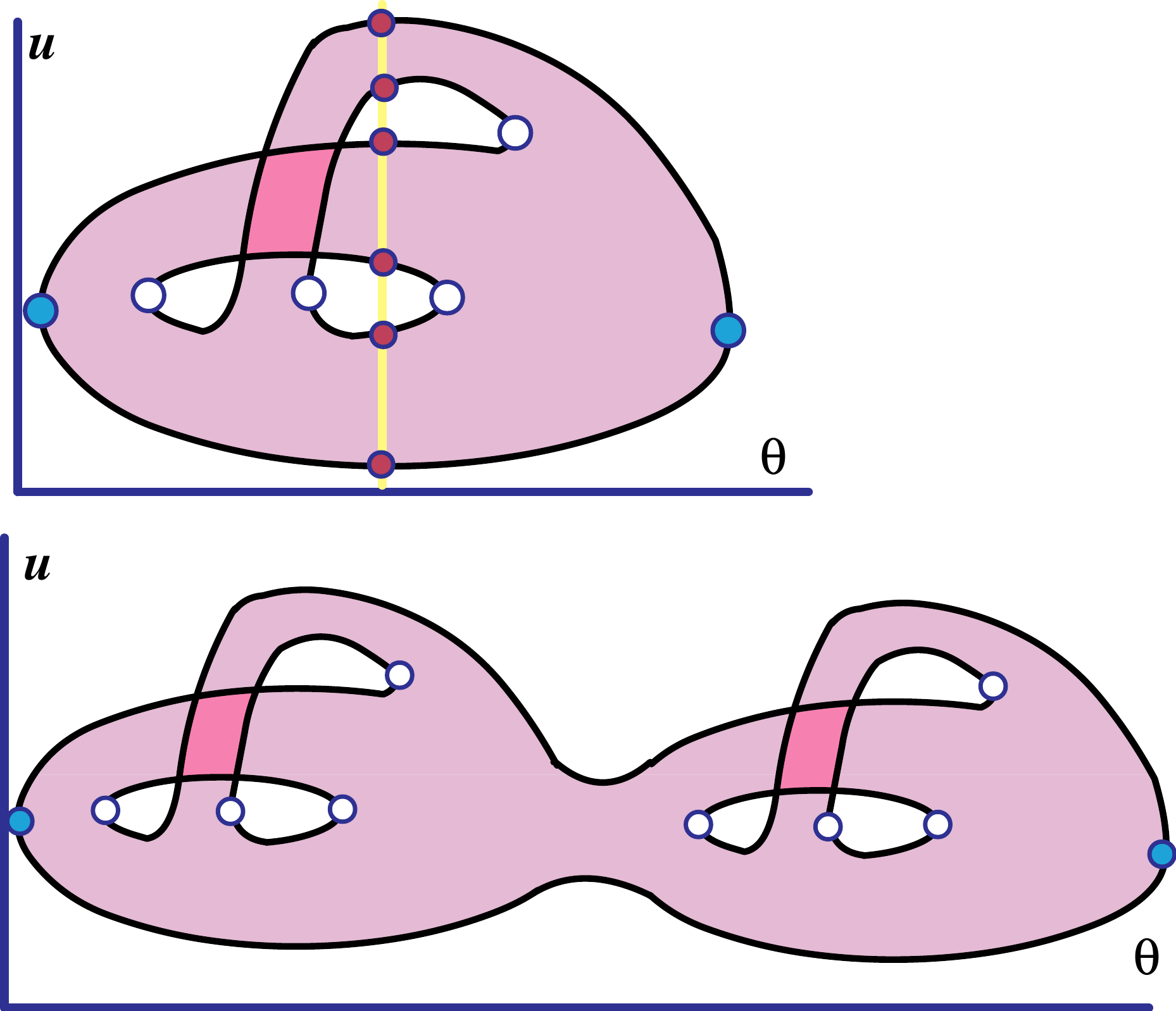}}
\bigskip
\caption{\small{Convex pseudo-envelops $\a: (X, \a^\dagger(\d_u)) \to (\hat X, \d_u)$ of a punctured torus $X$ (on the top), and of a punctured surface $X$ of genus $2$ (on the bottom). The envelop $\hat X$ is the cylinder $[0, 1] \times S^1$, equipped with the constant vector field $\d_u$. Both submersions $\a$ are generic relative to the vertical vector field $\d_u$ on the cylinder. The dots mark the multiplicity $2$ tangencies of $\a(\d X)$ to the $\theta$-fibers; the self-intersections of $\a(\d X)$ are unmarked. In both examples, the cardinality of the fibers $\theta \circ \a^\d: \d X \to S^1$ does not exceed $6$. }}
\label{fig.A1}
\end{figure}

The existence of a convex pseudo-envelop provides the $v$-flow on $X$ with a ``{\sf virtual global section}", a significantly effective tool.  Its existence allows us to transfer the results from \cite{K9} about immersions/embeddings $\{\b: M \to \R \times Y\}$, where $M$ and $Y$ are compact smooth $n$-manifolds, to the parallel results about convex envelops of traversing flows.

Next, we organize convex pseudo-envelops $\{\a: (X, v) \to (\hat X, \hat v)\}$ with the fixed target $(\hat X, \hat v)$ into the {\sf quasitopy} equivalence classes $\mathcal{QT}_d(\hat X, \hat v; \mathbf c\Theta)$, where $d$ is a given even natural number. For a typical $\hat v$-trajectory $\hat\g \subset \hat X$, $d$ gives an upper bound of the cardinality of the set  $\hat \g \cap \a(\d X)$.  Fig. \ref{fig.A_quasitopy} may give  some impression of  the nature of quasitopy;  it depicts a pair of submersions $\a_0, \a_1$, linked by a kind of cobordism with a strict combinatorial control of its tangent patterns.\smallskip

The quasitopies come in two flavors: $\mathcal{QT}_d^{\mathsf{emb}}(\hat X, \hat v; \mathbf c\Theta)$, formed by {\sf regular embeddings} $\a$, and $\mathcal{QT}_d^{\mathsf{sub}}(\hat X, \hat v; \mathbf c\Theta)$, formed by more general {\sf submersions} $\a$. 

Consider the locus $\d_1^+\hat X(\hat v) \subset \d\hat X$, where $\hat v$ points inside of $\hat X$, and the locus $\d_2^-\hat X(\hat v) \subset \d_1^+\hat X(\hat v)$, where $\hat v$ is tangent to $\d\hat X$. Then we construct two classifying maps
\begin{eqnarray}\label{eq.A_class} 
\Phi^{\mathsf{emb}}: \mathcal{QT}_d^{\mathsf{emb}}(\hat X, \hat v; \mathbf c\Theta) \to [(\d_1^+\hat X(\hat v), \d_2^-\hat X(\hat v)),\, (\mathcal P_d^{\mathbf c\Theta}, pt)],\\
\Phi^{\mathsf{sub}}: \mathcal{QT}_d^{\mathsf{sub}}(\hat X, \hat v; \mathbf c\Theta) \to [(\d_1^+\hat X(\hat v), \d_2^-\hat X(\hat v)),\, (\mathcal P_d^{\mathbf c\Theta}, pt)], \nonumber
\end{eqnarray}
whose targets are sets of homotopy classes from the quotient space $\d_1^+\hat X(\hat v)\big/ \d_2^-\hat X(\hat v)$ to the universal space $\mathcal P_d^{\mathbf c\Theta}$---the ``Grassmanian"---, formed by real polynomials of degree $d$ whose real divisors avoid a given poset $\Theta$.    

Then we prove (see Theorem \ref{th.envelops}) that $\Phi^{\mathsf{emb}}$ is a {\sf bijection} and $\Phi^{\mathsf{sub}}$ is a {\sf split surjection}.  

For a constant vector field $\hat v$ on the standard ball $D^{n+1}$, the sets  
$\mathcal{QT}_d^{\mathsf{sub/emb}}(D^{n+1}, \hat v; \mathbf c\Theta)$ are {\sf groups}; the group operation is induced by the boundary connected sum of pseudo-envelops/envelops.  These groups are abelian for $n > 1$. The group $\mathcal{QT}_d^{\mathsf{sub}}(D^{n+1}, \hat v; \mathbf c\Theta)$ is a split extension of $\mathcal{QT}_d^{\mathsf{emb}}(D^{n+1}, \hat v; \mathbf c\Theta)$ (see Theorem \ref{th.envelops on disks}).
We prove that $$\mathcal{QT}_d^{\mathsf{emb}}(D^{n+1}, \hat v; \mathbf c\Theta) \approx \pi_n(\mathcal P_d^{\mathbf c\Theta}, pt).$$ 
In the process, in the Subsection 4.2, we notice an interesting relation between the connected components of the of the space of convex traversing flows on the standard $(n+1)$-ball and the monoid $\mathbf{HS}_{n-1}$ of smooth types of {\sf integral homology $(n-1)$-spheres}, being considered up to connected sums with smooth {\sf homotopy} $(n-1)$-spheres (see Proposition \ref{prop.homology_sphere}).

The groups $\mathcal{QT}_d^{\mathsf{sub/emb}}(D^{n+1}, \hat v; \mathbf c\Theta)$ act on the quasitopies of convex pseudo-envelops \hfill\break  $\mathcal{QT}_d^{\mathsf{sub/emb}}(\hat X, \hat v; \mathbf c\Theta)$, and the classifying maps $\Phi^{\mathsf{sub/emb}}$ in (\ref{eq.A_class}) are equivariant. \smallskip

These classifying maps deliver a variety of (co)homotopy and (co)homology invariants of traversing flows which admit convex pseudo-envelops (see Theorem \ref{th.A_car_classes}, Proposition \ref{prop.Vass}, and Theorem \ref{th.counting_trajectories}). \smallskip

We also establish several of results (Theorem \ref{th.envelop_stabilization}, Corollary \ref{cor.STABILIZATION}) about the stabilization of the sets $\mathcal{QT}_d^{\mathsf{emb}}(\hat X, \hat v; \mathbf c\Theta)$ as $d \to \infty$.\smallskip

Finally,  we compute $\mathcal{QT}_d^{\mathsf{emb}}(\hat X, \hat v; \mathbf c\Theta)$ for may special cases of $\Theta$ and $(\hat X, \hat v)$.


\section{Spaces of real polynomials with constrained real divisors}
For our reader convenience, we state here a number of results from \cite{KSW1} and \cite{KSW2} about the topology of spaces of real monic univariate polynomials with constrained real divisors. These results are crucial for the applications to follow. To make the present paper even partially self-contained, we are basically recycling Section 2 from \cite{K9}. \smallskip

Let $\cP_d$ denote the space of real monic univariate polynomials of degree $d$. Given a polynomial $P(u) = u^d + a_{d-1}u^{d-1} + \cdots + a_0$ with real coefficients,  we consider its \emph{real divisor} $D_\R(P)$. 
Let $\om_i$ denotes the multiplicity of the $i$-th real root of $P$, the real roots being ordered by their magnitude.  We call the ordered $\ell$-tuple of natural numbers
$\om = (\omega_1 ,\ldots, \omega_\ell)$  
the {\sf real root multiplicity pattern} of $P(u)$, or the {\sf multiplicity pattern} for short.


Such sequences $\om$ form a {\sf universal poset} $(\mathbf\Om, \succ)$. The partial order ``$\succ$" in $\mathbf\Om$ is defined in terms of two types of elementary operations: {\sf merges} $\{\mathsf M_i\}_i$ and {\sf inserts} $\{\mathsf I_i\}_i$ 
The operation $\mathsf M_i$ merges a pair of adjacent entries $\om_i, \om_{i+1}$ of $\om = (\om_1, \dots, \om_i, \om_{i+1}, \dots, \om_q)$ into a single component $\tilde\om_i = \om_i + \om_{i+1}$, thus forming a new shorter sequence $\mathsf M_i(\om) = (\om_1, \dots, \tilde\om_i, \dots, \om_q)$. The insert operation $\mathsf I_i$ either insert $2$ in-between $\om_i$ and $\om_{i+1}$, thus forming a new longer sequence $\mathsf I_i(\om) = (\dots, \om_i, 2, \om_{i+1}, \dots)$, or, in the case of $\mathsf I_0$, appends $2$ before the sequence $\om$, or,  in the case $\mathsf I_q$,  appends $2$ after the sequence $\om$. 





We define the order $\om \succ \om'$, where $\om, \om' \in \mathbf \Om$, if one can produce $\om'$ from $\om$ by applying a sequence of these elementary operations.  
\smallskip

For a sequence $\omega = (\omega_1, \omega_2, \, \dots \, ,  \omega_q) \in \mathbf\Om$, we introduce the {\sf norm} and the {\sf reduced norm} of $\omega$ by the formulas: 
\begin{eqnarray}\label{eq2.2}
|\omega| =_{\mathsf{def}} \sum_i\; \omega_i \quad \text{and} \quad |\omega|'  =_{\mathsf{def}} \sum_i\; (\omega_i -1).
\end{eqnarray}
Note that $q$, the cardinality of the {\sf support} of $\omega$, is equal to $|\omega| - |\omega|'$.
\smallskip

Let $\mathring{\sR}^\omega_d$ be of the set of all polynomials with the real root multiplicity pattern $\omega$, and let $\sR_d^\omega$ be its closure in $\cP_d$.

 For a given collection $\Theta$ of multiplicity patterns $\{\om\}$, which share the pairity of their norms $\{|\om|\}$ and is closed under the merge and insert operations, we consider the union $\cP_d^\Theta$ of the subspaces $\mathring{\sR}_d^{\omega}$, taken over all $\omega \in \Theta$ such that $|\om| \leq d$ and $|\om| \equiv  d \mod 2$. We denote by $\cP_d^{\mathbf c\Theta}$ its {\sf complement} $\cP_d \setminus \cP^\Theta_d$. \smallskip

Since $\cP^\Theta_d$ is contractible, it makes more sense to consider its one-point compactification $\bar{\cP}_d^\Theta$. 
If the set $\bar{\cP}_d^{\Theta}$ is closed in $\bar{\cP}_d$, by the Alexander duality on the sphere  $\bar{\cP}_d \cong S^{d}$, we get
$$H^j(\cP_d^{\mathbf c\Theta}; \Z) \approx H_{d - j -1}(\bar{\cP}_d^{\Theta}; \Z).$$
This implies that  the spaces 
$\cP_d^{\mathbf c\Theta}$ and $\bar{\cP}_d^{\Theta}$ 
carry the same (co)homological information. Let us describe it in \emph{pure combinatorial terms}. \smallskip

Let us consider the following domain in $\R^{d+1}$:
\begin{eqnarray}\label{eq.E}
\mathcal E_d & =_{\mathsf{def}} & \big\{(u, P) \in  \R \times \mathcal P_d |\; P(u) \leq 0\big\} \text{\; and its boundary} \nonumber \\
\d\mathcal E_d & =_{\mathsf{def}} & \big\{(u, P) \in  \R \times \mathcal P_d |\; P(u) = 0\big\}. 
\end{eqnarray}
The pair $(\mathcal E_d, \d\mathcal E_d)$ is of a fundamental importance for us.
\smallskip

For a  subposet $\Theta \subset \bfOm$ and natural numbers $d, k$, we introduce the following notations:
\begin{eqnarray}\label{eq.Theta} 
\Theta_{[d]} =_{\mathsf{def}} &\{\om \in \Theta : \; |\om| = d \},\quad
\Theta_{\langle d]} =_{\mathsf{def}} \{\om \in \Theta : \; |\om| \leq d\}, \\ 
\Theta_{|\sim|' =k} =_{\mathsf{def}} & \{\om \in \Theta : \; |\om|' = k \}, \quad \Theta_{|\sim|' \geq k} =_{\mathsf{def}} \{\om \in \Theta : \; |\om|' \geq k \}.  \nonumber
 \end{eqnarray}
 
 Assuming that $\Theta\subset \bfOm$ is a \emph{closed} sub-poset, let 
$$\mathbf c\Theta =_{\mathsf{def}} \mathbf\Om \setminus \Theta \text{\; and \;} \mathbf c \Theta_{\langle d]} =_{\mathsf{def}} \bfOm_{\langle d]} \setminus \Theta_{\langle d]}.$$


We denote by $\Z[\bfOm_{\langle d]}]$ the free $\Z$-module, generated by the elements of $\bfOm_{\langle d]}$. Using the merge operators $\sM_k$ and the insert operators $\sI_k$  from \cite{KSW2}  on $\bfOm$,
we define  homomorphisms $\d_{\sM}: \Z[\bfOm_{\langle d]}] \to \Z[\bfOm_{\langle d]}]$, $\d_{\sI}: \Z[\bfOm_{\langle d]}] \to \Z[\bfOm_{\langle d]}]$ by 
$$\d_{\sM} (\omega) =_{\mathsf{def}} - \sum_{k=1}^{s_\omega-1} (-1)^k \sM_k(\omega) \, \text{  and   }\, \partial_{\sI} (\omega) =_{\mathsf{def}}  \sum_{k=0}^{s_\omega} (-1)^k \sI_k(\omega),$$
where $s_\omega =_{\mathsf{def}} |\omega| - |\omega|'$. In fact, $\d_{\sM}$ and $\d_{\sI}$ are anti-commuting differentials \cite{KSW2}.
Therefore, the sum
\begin{eqnarray}\label{eq.d+d}
\d  =_{\mathsf{def}} \d_{\sM} +\d_{\sI} : \Z[\bfOm_{\langle d]}] \to \Z[\bfOm_{\langle d]}]
\end{eqnarray}
 is a differential. 
%
\smallskip

For a closed poset $\Theta$, the restrictions of the operators $\d, \d_{\sM}$, and $\d_{\sI}$ to the free $\Z$-module $\Z[\Theta]$ are well-defined. Thus, for any closed subposet  $\Theta_{\langle d]} \subset \mathbf\Om_{\langle d]}$, we may consider the differential complex $\d: \Z[\Theta_{\langle d]}] \to \Z[\Theta_{\langle d]}]$, whose $(d-j)$-grading is defined by the module $\Z[\Theta_{\langle d]} \cap \Theta_{|\sim|' = j}]$. 
We denote by $\d^\ast: \Z[\Theta_{\langle d]}]^\ast \to \Z[\Theta_{\langle d]}]^\ast$ its dual operator, where $\Z[\Theta_{\langle d]}]^\ast =_{\mathsf{def}} \mathsf{Hom}(\Z[\Theta_{\langle d]}], \Z)$.
\smallskip

Then we consider the quotient set $\mathbf\Theta^\#_{\langle d]} := \mathbf\Om_{\langle d]}\big / \Theta_{\langle d]}$. For the closed subposet $\Theta_{\langle d]}$, the partial order in $\mathbf\Om_{\langle d]}$ induces a partial order in the quotient $\mathbf\Theta^\#_{\langle d]}$.

Finally, we introduce a new differential complex $(\Z[\mathbf\Theta^\#_{\langle d]}], \d^\#)$ by including it in the short exact sequence of differential complexes:
\begin{eqnarray}\label{eq.quotient_complex}
0 \to (\Z[\Theta_{\langle d]}], \d) \to (\Z[\mathbf\Om_{\langle d]}], \d) \to (\Z[\mathbf\Theta^\#_{\langle d]}], \d^\#) \to 0.
\end{eqnarray}

We will rely on the following result from \cite{KSW2}, which reduces the computation of the reduced cohomology $\tilde H^\ast(\cP_d^{\mathbf c\Theta}; \Z)$ to Algebra and Combinatorics. 

\begin{theorem}\label{thA}{\bf (\cite{KSW2})} 
  Let $\Theta \subset \bfOm_{\langle d]}$ be a closed subposet. 
  Then, for any $j \in [0, d]$, we get group isomorphisms
   \begin{eqnarray}\label{eq.3.5A} 
    \tilde H^j(\cP_d^{\mathbf c\Theta}; \Z)\;  \stackrel{\mathcal D}{\approx} \;H_{d-j}(\bar{\cP}_d, \bar{\cP}_d^{\Theta}; \Z)\; 
    \approx  H_{d- j}\big(\d^\#:  \Z[\mathbf\Theta^\#] \to \Z[\mathbf\Theta^\#]\big), \\
  \quad \tilde  H_j(\cP_d^{\mathbf c\Theta}; \Z)\;  \stackrel{\mathcal D}{\approx} \;H^{d-j}(\bar{\cP}_d, \bar{\cP}_d^{\Theta}; \Z)\; 
    \approx  H_{d- j}\big((\d^\#)^\ast:  (\Z[\mathbf\Theta^\#])^\ast \to (\Z[\mathbf\Theta^\#])^\ast\big), \nonumber
      \end{eqnarray}
      where $\mathcal D$ is the Poincar\'{e} duality isomorphism. \hfill $\diamondsuit$
 \end{theorem}

Consider the embedding $\e_{d, d+2}: \mathcal P_d \to \mathcal P_{d+2}$, defined by the formula
\begin{eqnarray}\label{eq.stable} 
\e_{d, d+2}(P)(u) =_{\mathsf{def}} (u^2+1) \cdot P(u).
\end{eqnarray}
It preserves the 
$\bfOm$-stratifications of  the two spaces by the combinatorial types $\om$ of real divisors. The embedding $\e_{d, d+2}$ makes it possible to talk about  \emph{stabilization of the homology/cohomology} of the spaces $\bar{\mathcal P}_d^\Theta$ and $\mathcal P_d^{\mathbf c\Theta}$, as $d \to \infty$. \smallskip

For a closed poset $\Theta \subset \bfOm$, consider the closed finite poset $\Theta_{\langle d]} = \bfOm_{\langle d]} \cap \Theta$. It is generated by some 
\emph{maximal} elements $\omega_\star^{(1)}, \dots , \omega_\star^{(\ell)}$, where $\ell = \ell(d)$.  We introduce two useful quantities: 
\begin{eqnarray}\label{eq.eta(Theta)}
\eta_\Theta(d) & =_{\mathsf{def}} & \max_{i \,\in \,[1,\, \ell(d)]} \big\{\big(|\om_\star^{(i)}| - 2|\om_\star^{(i)}|'\big)\big\}, \\
\psi_\Theta(d) & =_{\mathsf{def}} & \frac{1}{2}\big(d + \eta_\Theta(d)\big).
\end{eqnarray}

Note that $$\psi_\Theta(d) =  \frac{1}{2}\;\min_{i \, \in \, [1,\ell(d)]}\big\{ (d- |\om_\star^{(i)}|') + (|\om_\star^{(i)}|-|\om_\star^{(i)}|')\big\} < d,$$ 
where both summands, the ``codimension" $(d- |\om_\star^{(i)}|')$ and the ``support" $(|\om_\star^{(i)}|-|\om_\star^{(i)}|')$, are positive and each does not exceed $d$. At the same time,
$\eta_\Theta(d)$ may be negative. \smallskip

In the stabilization results about quasitopies,  the quantity 
\begin{eqnarray}\label{eq.xi(Theta)}
\xi_\Theta(d+2) =_{\mathsf{def}} d+2 - \psi_\Theta(d+2).
\end{eqnarray}
%
plays the key role. 
 \smallskip

Now we are in position to state the main stabilization result from \cite{KSW2}: 

\begin{theorem}\label{th.main_stab} {\bf (short stabilization: $\mathbf{\{d \Rightarrow d+2\}}$)} Let $\Theta$ be a closed subposet of $\mathbf\Om$. Let the embedding $\e_{d, d+2}:\, \mathcal P_{d} \subset \mathcal P_{d+2}$ be as in (\ref{eq.stable}). \begin{itemize}
\item Then, for all  $j \geq  \psi_\Theta(d+2) -1$, we get a homological isomorphism 
$$(\e_{d, d+2})_\ast: H_j(\bar{\mathcal P}_{d}^\Theta; \Z) \approx H_{j+2}(\bar{\mathcal P}_{d+2}^{\Theta}; \Z),$$

\item and, for all $j \leq d+2 - \psi_\Theta(d+2)$, 
a homological isomorphism  $$(\e_{d, d+2})_\ast: H_j( \mathcal P_{d}^{\mathbf c\Theta}; \Z) \approx H_j(\mathcal P_{d+2}^{\mathbf c\Theta}; \Z). \quad \quad \diamondsuit$$
\end{itemize}
\end{theorem}

\begin{definition}\label{def.profinite}
A closed poset $\Theta\subseteq \mathbf\Om$ is called {\sf profinite} if, for all  integers $q \geq 0$, there exist only finitely many elements $\om \in \Theta$ such that $|\om|' \leq q$. 
  \hfill $\diamondsuit$
\end{definition}

\begin{corollary}\label{cor.stable_homology}{\bf (long stabilization: $\mathbf{\{d \Rightarrow \infty\}}$)} For any closed profinite poset $\Theta\subset \mathbf\Omega$,  
and for each $j$, the homomorphism
$$(\e_{d,d'})_\ast:\; H_j( \mathcal P_{d}^{\mathbf c\Theta}; \Z) \approx H_j(\mathcal P_{d'}^{\mathbf c\Theta}; \Z)$$
is an isomorphism for all sufficiently big $d \leq d'$, $d \equiv d' \mod 2$.  \hfill  $\diamondsuit$
\end{corollary}

As a result, 
we may talk about the {\sf stable homology} $H_j(\mathcal P_{\infty}^{\mathbf c\Theta}; \Z)$, the direct limit  \hfill \break $\lim_{d \to \infty}H_j( \mathcal P_{d}^{\mathbf c\Theta}; \Z)$.
\smallskip

Let us describe a few special cases of stabilization from \cite{KSW2}. 
%
For $k \geq 1$,  $q \in [0, d]$, and $q \equiv d \mod 2$, let us consider the closed poset
\begin{eqnarray}\label{Theta-wedge} 
\bfOm_{|\sim|^{'} \geq k}^{(q)} =_{\mathsf{def}}\, \big\{ \om \in \bfOm_{\langle d]} \text{ such that } |\omega|' \geq k \text{ and } |\om| \geq q\big\}.
\end{eqnarray}
Note that, for $\Theta = \bfOm_{|\sim|^{'} \geq k}^{(0)} =_{\mathsf{def}} \bfOm_{|\sim|' \geq k}$, the space $\bar{\cP}_d^\Theta$ is the entire $(d-k)$-skeleton of $\bar{\cP}_d$. 
 
 \begin{eqnarray}\label{eq.A-bouquet} 
 \text{ Let \;} A(d, k, q) =_{\mathsf{def}}\; \big|\chi\big((\Z\big[\bfOm_{|\sim|^{'} \geq k}^{(q)}\big], \d)\big)\big| = \big| \chi\big(\bar{\mathcal P}_{d}^{\bfOm_{|\sim|^{'} \geq k}^{(q)}}\big) -1 \big |,
 \end{eqnarray}
 the absolute value of the Euler number of the differential complex $\big(\Z\big[\bfOm_{|\sim|^{'} \geq k}^{(q)}\big], \d\big)$.

%

%

\begin{proposition}{\bf (\cite{KSW2})} 
  \label{prop.skeleton}
  Fix $k \in [1, d)$ and $q \ge 0$ such that $q \equiv d \mod 2$. Let $\Theta = \bfOm_{|\sim|^{'} \geq k}^{(q)}$. 
  
  Then the one-point compactification 
  $\bar \cP^{\Theta}_d$ has the homotopy type of 
	a bouquet of $(d-k)$-dimensional spheres. The number of spheres in the bouquet
	equals $A(d, k, q)$.\hfill $\diamondsuit$
\end{proposition}

%




\begin{proposition}{\bf (\cite{KSW1})} 
  \label{prop.fundamental_groups}
  Let $\Theta \subseteq \paOm{\langle d]}{\geq 2}$ 
  be a closed poset. For $d' \geq d+2$ such that $d'  \equiv d \mod 2$, 
  let $\hat\Theta_{\{d'\}}$ be the 
  smallest closed poset in $\bfOm_{\langle {d'}]}$ containing
  $\Theta$. 
  
  Then for $d' \geq d+2$, we have an isomorphism 
	$\pi_1\big(\cP_{d'}^{\mathbf c \hat \Theta_{\{d'\}}}\big) \cong \pi_1\big(\cP_{d+2}^{\mathbf c\Theta_{\{d+2\}}}\big)$ of the fundamental groups. \hfill $\diamondsuit$
\end{proposition}





 
\section{Submersions \& embeddings of manifolds whose boundary has constrained tangency patterns to the product $1$-foliations}

This section forms a bridge between the results from \cite{K9} and our main results from Section 4. The reader may choose to surf Section 3 or to proceed directly to Section 4.  Since all the results of this section are instant derivations of the similar results from \cite{K9}, we provide just an outline of their validations. Although we will not use directly the results of Section 3  in Section 4, this section may induce the ``right mindset" for the reader. \smallskip

Let $Y$ be a smooth compact $n$-manifold. 
Having in mind applications to traversing vector flows, we move away from immersions and embeddings $\b: M \to \R \times Y$ of $n$-manifolds $M$, the topic of \cite{K9}, to {\sf submersions} and {\sf regular embeddings} $\a: X \to \R \times Y$ of compact smooth $(n+1)$-manifolds $X$ with boundary into the product  $\R \times Y$.  When $\dim(X) = \dim(Y)+1$, the submersions and immersions are the same notion. Moreover, the restriction of a submersion $\a$ to $\d X$ is an immersion. 
Of cause, if $\a:  X \to \R \times Y$ is an embedding, so is $\a^\d := \a |: \d X \to \R \times Y$.

Therefore many constructions and notions from \cite{K9}, with the help of the correspondence $\a \leadsto \a^\d$, apply instantly to submersions  $\a: X \to \R \times Y$ such that $\dim X = \dim Y +1$.
 
\begin{remark}
\emph{From the viewpoint of this paper, the main difference between immersions $\b: M \to \R \times Y$ and submersions $\a: X \to \R \times Y$, where $\dim X = \dim M + 1$, is that not any $\b$ is a boundary $\a^\d$ of some $\a$. For example, the figure $\infty$ in the plane does not bound a submersion of a $2$-manifold. See \cite{Pa} for the comprehensive theory of possible extensions of a given immersion $\b$ to a submersion $\a$.} 
\hfill $\diamondsuit$
\end{remark}

\begin{example}\label{ex.submersions}
\emph{The following simple construction provides \emph{models of submersions} that animate our treatment. Let $W$ be a codimension zero compact submanifold of a given manifold $V$. Consider a covering map $\pi: \tilde W \to W$ with a finite fiber. Let $X \subset \tilde W$ be a compact codimension zero submanifold. It is possible to isotop the imbedding $X \subset \tilde W$ so that $\pi: \d X \to W$ will be an immersion with all the multiple crossings of $\pi(\d X)$ being in general position. Of course, each crossing has the multiplicity that does not exceed the cardinality of the $\pi$-fiber. Then $\pi: X \to V$ is the model example of a submersion to keep in mind.} \hfill $\diamondsuit$
\end{example}



Let us introduce the central to this paper notion of {\sf quasitopy for submersions} $\a: X \to \R \times Y$, an analogue of Definition 3.7 from \cite{K9}. 
We fix a natural number $d$ and consider a closed sub-poset $\Theta \subset \mathbf\Om$  such that $(\emptyset) \notin \Theta$ (the, so called, $\Lambda$-{\sf condition} (3.10) from \cite{K9}).  
For topological reasons, we will consider only the case $d \equiv 0 \mod 2$.
\smallskip

Let $\mathcal L$ be the $1$-foliation of $\R \times Y$ by the fibers of the obvious projection $\R \times Y \to Y$, and $\mathcal L^\bullet$ be the $1$-foliation of $\R \times Y \times [0, 1]$ by the fibers of $\R \times Y \times [0, 1] \to Y \times [0, 1]$.\smallskip

Let $X$ be a compact 
smooth $(n+1)$-dimensional manifold with boundary. Consider a smooth map  $\a: X \to \R \times Y$ such that:
\begin{eqnarray}\label{multiplicity_condition}
\end{eqnarray}
\begin{itemize}
\item $\a$ is a {\sf submersion}, 
\smallskip

\item for each $y \in Y$, the total multiplicity $m_\b(y)$ of $\a(\d X)$ with respect to the foliation $\mathcal L$ (see \cite{K9}, formula (3.4)) is less than or equal to $d$ and $m_\b(y) \equiv d \mod 2$,\smallskip

 \item for each $y \in Y$, the combinatorial tangency pattern $\omega^{\a}(y) \in \mathbf\Om_{\langle d]}$ of $\a(\d X)$ with respect to $\mathcal L$ does not belong to $\Theta$,\smallskip
 
 \item for each $y \in \d Y$,   $\omega^\a(y) = (\emptyset)$.
\end{itemize}
\smallskip

Note that the normal bundle $\nu^\a$ to $\a(\d X)$ in $\R \times Y$ is trivial. 
\smallskip 

Let $X_0, X_1$ be two compact smooth (oriented) $(n+1)$-dimensional manifolds with boundary. We consider a compact smooth (oriented) $(n+2)$-manifold $W$ with conners $\d X_0 \coprod \d X_1$ such that $\d W = (X_0 \coprod X_1) \bigcup_{\{\d X_0 \coprod \d X_1\}} \delta W$, where $\delta W$ is a smooth (oriented) cobordism between $\d X_0$ and $\d X_1$. Let $Z =_{\mathsf{def}} \R \times Y$ and $Z^\d =_{\mathsf{def}} \R \times \d Y$. 

Let  $A: W \to \R \times Z$, where $\dim(W) = \dim(Z) +1$, be a {\sf submersion}. In particular, $A: \delta W \to \R \times Z, \; A|_{\d X_0 } \to \R \times (Y \times \{0\})$, and $A|_{\d X_1 } \to \R \times (Y \times \{1\})$ are immersions.  
\smallskip

The next two definitions lay down the foundation for notions of quasitopy 
of traversing vector fields, the main subject of Section 4 (see Fig. 2). 

\begin{definition}\label{quasi_isotopy_of_sub}
Let us fix natural numbers $d' \geq d$, $d' \equiv d \equiv 0 \mod 2$.  Consider closed subposets $\Theta' \subset \Theta \subset  \mathbf\Om$ such that $(\emptyset) \notin \Theta$. \smallskip

We say that a two submersions $\a_0: X_0 \to \R \times Y$ and $\a_1: X_1 \to \R \times Y$ are $(d, d'; \mathbf c\Theta, \mathbf c\Theta')$-{\sf quasitopic}, if there exists a compact smooth 
$(n+2)$-manifold $W$ as above and a smooth submersion 
$A: W \to \R \times Z$ so that:
\begin{itemize} 
\item $A|_{X_0} = \a_0$ and $A|_{X_1} = \a_1$; 
\smallskip
  
\item for each $z \in Z$, the total multiplicity $m_A(z)$ of $A(\delta W)$ with respect to the fiber $\mathcal L^\bullet_z$ is such that $m_A(z) \leq d'$,  $m_A(z) \equiv d' \mod 2$, and the combinatorial tangency pattern $\omega^A(z)$ of $A(\delta W)$ with respect to $\mathcal L^\bullet_z$ belongs to $\mathbf c\Theta'$; 
\smallskip
  
\item for each $z \in Y \times (\{0\} \cup \{1\})$, the total multiplicity $m_A(z)$ of $A(\delta W)$ with respect to the fiber $\mathcal L_z$ is such that  $m_A(z) \leq d$, $m_A(z) \equiv d \mod 2$, and the combinatorial tangency pattern $\omega^A(z)$ of $A(\delta W)$ with respect to $\mathcal L$ belongs to $\mathbf c\Theta$. 
\smallskip 

\item for each $z \in Z^\d$, $\omega^A(z) = (\emptyset)$. 
\end{itemize}

We denote by $\underline{\overline{\mathsf{QT}}}^{\mathsf{\, sub/emb}}_{\,d, d'}(Y; \mathbf{c}\Theta; \mathbf{c}\Theta')$ the set of quasitopy classes of such submersions/ embeddings $\a: X \to  \R \times Y$. \hfill $\diamondsuit$
\end{definition}
It is easy to check that the quasitopy of submersions is an equivalence relation. Recall that in \cite{K9}, Definition 3.7, we have introduced a similar notion of quasitopy for immersions/embeddings $\b: (M, \d M) \to (\R \times Y, \R \times \d Y)$. There, we used the notation $$\mathsf{QT}^{\mathsf{imm/emb}}_{d, d'}(Y;  \mathbf{c}\Theta; \mathbf{c}\Theta') =_{\mathsf{def}} \mathcal{QT}^{\mathsf{imm/emb}}_{d, d'}(Y, \d Y; \hfill\break \mathbf{c}\Theta, (\emptyset); \mathbf{c}\Theta')$$ for the set of equivalence classes of immersions/embeddings $\b$ under the quasitopy relation.

As for immersions $\b: M \to \R \times Y$, for any choice of connected components $\kappa_1 \in \pi_0(\d Y_1), \kappa_2 \in \pi_0(\d Y_2)$, the connected sum operation (see \cite{K9}, formula (3.15))
\begin{eqnarray}\label{eq_cup_sub}
\uplus:\;\underline{\overline{\mathsf{QT}}}^{\mathsf{\, sub/emb}}_{\,d, d'}(Y_1; \mathbf{c}\Theta; \mathbf{c}\Theta')  \times \underline{\overline{\mathsf{QT}}}^{\mathsf{\, sub/emb}}_{\,d, d'}(Y_2; \mathbf{c}\Theta; \mathbf{c}\Theta') \to \nonumber \\
\to\underline{\overline{\mathsf{QT}}}^{\mathsf{\, sub/emb}}_{\,d, d'}(Y_1\#_\d Y_2; \mathbf{c}\Theta; \mathbf{c}\Theta'),
\end{eqnarray}
is well-defined for the quasitopies of submersions. 

It converts the set $\underline{\overline{\mathsf{QT}}}^{\mathsf{\, sub/emb}}_{\,d, d'}(D^{n}; \mathbf{c}\Theta; \mathbf{c}\Theta')$ into a {\sf group} 
$\underline{\overline{\mathsf{H}}}^{\mathsf{\, sub/emb}}_{\,d, d'}(n; \mathbf{c}\Theta; \mathbf{c}\Theta')$ (abelian for $n > 1$).  This group acts, via the connected sum operation $\uplus$, on the set $\underline{\overline{\mathsf{QT}}}^{\mathsf{\, sub/emb}}_{\,d, d'}(Y; \mathbf{c}\Theta; \mathbf{c}\Theta')$, provided that a connected component of $\d Y$ is chosen. To get a better insight, compare this claim with Proposition 3.2 from \cite{K9}.
\smallskip




For two pairs $X_1 \supset A_1$ and $X_2 \supset A_2$ of topological spaces, we denote by $[(X_1,A_1),\hfill\break (X_2, A_2)]$ the set of homotopy classes of continuous maps $g: X_1 \to X_2$,  where $g(A_1) \subset A_2$.\smallskip

\begin{definition}\label{def.triples}
Given three pairs of spaces $X_1 \supset A_1$,  $X_2 \supset A_2$, $X_3 \supset A_3$, and a fixed continuous map $\e: (X_2, A_2) \to (X_3, A_3)$, we denote by 
\begin{eqnarray}\label{homotopy_of_tripples}
[[(X_1, A_1),\, \e: (X_2, A_2) \to (X_3, A_3)]]
\end{eqnarray}
the set of homotopy classes $[g]$ of continuous maps $g: (X_1, A_1) \to (X_2, A_2)$, modulo the following equivalence relation: by definition, $[g_0] \sim [g_1]$, where $g_0: (X_1, A_1) \to (X_2, A_2)$ and $g_1: (X_1, A_1) \to (X_2, A_2)$ are continuous maps, if the compositions $\e\circ g_0$ and $\e \circ g_1$ are homotopic as maps from $(X_1, A_1) $ to $(X_3, A_3)$. 
\hfill  $\diamondsuit$
\end{definition}

Following the proof of Proposition 3.3 and Theorem 3.2 from \cite{K9}, we get Theorem \ref{th.section}, their analogue for submersions. It is crucial that here $d \equiv 0 \mod 2$, which implies that any regular embedding $\b: M \to \R \times Y$, such that all the multiplicities $\{m_\b(y)\}_{y \in Y}$ are even, bounds a regular embedding $\a: X \to \R \times Y$, where $\d X = M$ and $\a|_{\d X} = \b$. \smallskip

For an $n$-dimensional $Y$, the next lemma reduces the computation of quasitopies \hfill \break $\underline{\overline{\mathsf{QT}}}^{\mathsf{\, emb}}_{\,d, d'}(Y; \mathbf{c}\Theta; \mathbf{c}\Theta')$, based on regular embeddings $\a: X \to \R \times Y$ of $(n+1)$-dimensional manifolds $X$, to the computation of quasitopies $\mathsf{QT}^{\mathsf{emb}}_{d, d'}(Y; \mathbf{c}\Theta; \mathbf{c}\Theta')$, based on the regular embeddings $\b: M \to  \R \times Y$ of $n$-dimensional manifolds $M$.

\begin{lemma}\label{lem.bijection D} For closed posets $\Theta' \subset \Theta \subset \mathbf\Om$ such that $(\emptyset) \notin \Theta$ and $d \leq d'$, $d \equiv d' \equiv 0 \mod 2$, the  map $$\Delta:\,\underline{\overline{\mathsf{QT}}}^{\mathsf{\, emb}}_{d, d'}(Y; \mathbf{c}\Theta; \mathbf{c}\Theta') \to \mathsf{QT}^{\mathsf{emb}}_{d, d'}(Y; \mathbf{c}\Theta; \mathbf{c}\Theta')$$ that takes an embedding $\a: X \to \R \times Y$ to the embedding $\a^\d: \d X \to \R \times Y$ is a bijection.

As a result, the obvious map $\mathcal A^\bullet:\, \underline{\overline{\mathsf{QT}}}^{\mathsf{\, emb}}_{d, d'}(Y; \mathbf{c}\Theta; \mathbf{c}\Theta') \to \underline{\overline{\mathsf{QT}}}^{\mathsf{\, imm}}_{d, d'}(Y; \mathbf{c}\Theta; \mathbf{c}\Theta')$
is injective; in other words, if two embedding are quasitopic via a submersion, they are quasitopic via an embedding.  
\end{lemma}

\begin{proof} The main step is contained in the proof of Lemma 3.6 from \cite{K9}. Let us describe its flavor. Since, for $d \equiv 0 \mod 2$ and a closed  $n$-manifold $M$, any $(\d\mathcal E_d)$-{\sf regular} (see Definition 3.4 in \cite{K9}) embedding $\b: M \subset \R \times \text{int}(Y)$ bounds a (orientable when $Y$ is orientable) $(n+1)$-manifold $\a: X_\b \subset \R \times Y$, the map $\Delta$ is onto. Evidently, the combinatorial tangency types to $\mathcal L$ are determined by $\b(M)$. Thus, every embedding $\a \in \underline{\overline{\mathsf{QT}}}^{\mathsf{\, emb}}_{d, d'}(Y; \mathbf{c}\Theta; \mathbf{c}\Theta')$ produces an element $\a^\d$ in $\mathsf{QT}^{\mathsf{\, emb}}_{d, d'}(Y; \mathbf{c}\Theta; \mathbf{c}\Theta')$.  By the same token, if a $(\d\mathcal E_d)$-regular embedding $\b: M \subset \R \times \text{int}(Y)$ bounds a $(\d\mathcal E_{d'})$-regular embedding $B: N \subset \R \times Z$ (where $Z =_{\mathsf{def}} Y \times [0, 1]$), whose tangency to $\mathcal L^\bullet$ patterns belong to $\mathbf{c}\Theta'$, then $X_\b \cup_{\d M} N$ bounds (an orientable when $Y$ is orientable) $(n+2)$-manifold $W \subset \R \times Z$, provided $d' \equiv 0 \mod 2$. Here $N \subset \R\times Z$ is a compact $(n+1)$-manifold such that $\d N = M$ and $B|_{M} = \b$. Thus the map $\Delta$ is injective. By the previous argument, it is bijective.   

By \cite{K9}, Proposition 3.5, the map $\mathcal A: \mathsf{QT}^{\mathsf{\, emb}}_{d, d'}(Y; \mathbf{c}\Theta; \mathbf{c}\Theta') \to \mathsf{QT}^{\mathsf{\, imm}}_{d, d'}(Y; \mathbf{c}\Theta; \mathbf{c}\Theta')$ (see formula (3.32) in \cite{K9}) is injective. By the argument above, $\Delta$ is injective. Chasing the obvious square diagram, formed by the sources and targets of the maps $\mathcal A, \mathcal A^\bullet$, we conclude that $\mathcal A^\bullet$ is injective as well. 
\end{proof}

Combining Lemma \ref{lem.bijection D} with Theorem 3.2 from \cite{K9}, we get the following results.

\begin{theorem}\label{th.section} We fix even natural numbers $d' \geq d$, $d' \equiv d  \mod 2$,  and  closed subposets $\Theta' \subset \Theta \subset \mathbf\Om_{\langle d]}$ such that  $(\emptyset) \notin \Theta$. Let $Y$ be a smooth compact $n$-manifold.

Then any submersion $\a: X \to \R \times Y$ as in (\ref{multiplicity_condition}) gives rise\footnote{not in a canonical fashion} to a map $\Psi(\a): (Y, \d Y) \to (\cP_d^{\mathbf c\Theta}, \cP_d^{(\emptyset)})$. Moreover, $(d, d'; \mathbf c\Theta, \mathbf c\Theta')$-quasitopic submersions/embeddings $\a_0: X_0 \to \R \times Y$ and $\a_1: X_1 \to \R \times Y$ produce homotopic maps $\Psi(\a_0)$ and $\Psi(\a_1)$. 

In this way, we get a map 
$$\Psi_{d, d'}^{\mathsf{\, sub/emb}}: \; \underline{\overline{\mathsf{QT}}}^{\mathsf{\, sub/emb}}_{\,d, d'}(Y; \mathbf{c}\Theta; \mathbf{c}\Theta') \to \big[\big[(Y, \d Y),\; \e_{d, d'}: (\cP_d^{\mathbf c\Theta}, \cP_d^{(\emptyset)}) \to (\cP_{d'}^{\mathbf c\Theta'}, \cP_{d'}^{(\emptyset)}) \big]\big],  
$$

Conversely, the homotopy class of any continuous map $G: (Y, \d Y) \to  (\cP_d^{\mathbf c\Theta}, \cP_d^{(\emptyset)})$ 
 is realized by a smooth regular \emph{embedding} $\a: X \hookrightarrow \R \times Y$ which satisfies (\ref{multiplicity_condition}); that is, $G = \Psi_{d, d'}^{\mathsf{\, emb}}(\a)$. 
 \smallskip
 
Moreover, $\Psi_{d, d'}^{\mathsf{\, emb}}$ is a bijection,  and $\Psi_{d, d'}^{\mathsf{\, sub}}$ is a surjection, admitting a right inverse.
 \hfill $\diamondsuit$
\end{theorem}





\section{Convex envelops of traversing flows, their quasitopies, \&  characteristic classes}

\subsection{Traversing, generic, and convex vector fields.  Morse stratifications. Spaces of convex traversing  vector fields}

\begin{definition}\label{def.traversing} A  vector field $v \neq 0$ on a compact smooth manifold $X$ is called {\sf traversing}, if each $v$-trajectory is homeomorphic either to a closed interval, or to a point. 
\hfill $\diamondsuit$
\end{definition}

Let $v$ be a  traversing and {\sf boundary generic} (see \cite{K1}, \cite{K2}, and Definition \ref{def.boundary_generic} below) vector field on a compact smooth $(n+1)$-manifold $X$ with boundary.  As we will see soon, every trajectory $\g$ of such a vector field $v$ generates its {\sf  tangency divisor} $D_\g$, an ordered sequence of points in $\g$, together with their multiplicities (natural numbers).\smallskip


We try to ``go around" the fundamental discontinuity of   the map $x \to \g_x \to D_{\g_x}$,  where $\g_x$ stands for the $v$-trajectory through $x \in X$. This requires ``to envelop" the pair $(X, v)$ in a {\sf convex envelop/pseudo-envelop} $(\hat X, \hat v)$ (see Definition \ref{def.convex_envelop}). 
The  convex pseudo-envelops, when available,  will greatly simplify our analysis of traversing flows. In the spirit of Section 3, we will apply our results about immersions and submersions (against the background of product $1$-foliations) from \cite{K9}  to the new environment of convex envelops of traversing flows.\smallskip

Following \cite{Mo}, for a generic vector field $v$ on a smooth compact $(n+1)$-dimensional manifold $X$, such that $v \neq 0$ along $\d X$, let us describe an important {\sf Morse stratification} $\{\d_j^\pm X(v)\}_{j \in [1, \dim X]}$ of the boundary $\d X$. The stratum $\d_jX =_{\mathsf{def}} \d_jX(v)$ has the following description (see \cite{K1}) in terms of an auxiliary function $z: \hat X \to \R$ that satisfies the three properties:
\begin{eqnarray}\label{eq2.3}
\end{eqnarray}

\begin{itemize}
\item $0$ is a regular value of $z$,   
\item $z^{-1}(0) = \d X$, and 
\item $z^{-1}((-\infty, 0]) = X$. 
\end{itemize}

In terms of $z$, the locus $\d_jX =_{\mathsf{def}} \d_jX(v)$ is defined by the equations: 
$$\big\{z =0,\; \mathcal L_vz = 0,\; \ldots, \;  \mathcal L_v^{(j-1)}z = 0\big\},$$
where $\mathcal L_v^{(k)}$ stands for the $k$-th iteration of the Lie derivative operator $\mathcal L_v$ in the direction of $v$ (see \cite{K2}). 
The pure stratum $\d_jX^\circ \subset \d_jX$ is defined by the additional constraint  $\mathcal L_v^{(j)}z \neq 0$. The locus $\d_jX$ is the union of two loci: {\bf (1)} $\d_j^+X$, defined by the constraint  $\mathcal L_v^{(j)}z \geq  0$, and {\bf (2)} $\d_j^-X$, defined by the constraint  $\mathcal L_v^{(j)}z \leq  0$. The two loci, $\d_j^+X$ and $\d_j^-X$, share a common boundary $\d_{j+1}X$. 

For a {\sf generic} $v$, all the strata $\d_j X$ are smooth $(n+1-j)$-manifolds. The requirement of $v$ being generic with respect to $\d X$ may be expressed as the property of the $j$-form
\begin{eqnarray} \label{eq.generic_simple}
dz \wedge d(\mathcal L_vz) \wedge \; \ldots \;  \wedge\; d(\mathcal L_v^{(j-1)}z)
\end{eqnarray}
 being a nonzero section of the bundle $\bigwedge^{j}T_\ast X$ along the locus  ${\d_jX}$ \;  for all $j \in [1, n+1]$.
If $v$ on $X$ is generic to $\d X$, then each point $x \in \d X$ belongs to a unique  minimal stratum $\d_j X \subset \d X$ with a maximal $j = j(x) \leq n+1$.  In the generic case, at each $b \in \d X$,  a flag $$\mathsf{Flag}_b(v) ={\mathsf{def}} \{T_b(\d X) = F^n \supset F^{n-1} \supset \ldots \supset F^{n-j(b)+1}\}$$  is generated by the tangent spaces at $b$ to all the Morse strata $\{\d_j X\}_{j \leq j(b)}$ that contain $b$.   

Let $\hat v$ be a traversing vector field on a compact smooth $(n+1)$-dimensional manifold $\hat X$ with boundary. Consider a submersion $\a: X \to \mathsf{int}(\hat X)$, $\dim X = \dim \hat X$, such that the self-intersections of $\a(\d X)$ mutually transversal. Let $v = \a^\dagger(\hat v)$ be the transfer of $\hat v$ to $X$.\smallskip

For general submersions $\a$, which are not necessarily embeddings, the situation is more complex: not only one gets multiple self-intersections $\{\Sigma_k\}_{k \in [2, n+1]}$ of various branches of $\a(\d X)$, but such self-intersections may be tangent to the $\hat v$-flow in a variety of ways that produce similar Morse-type stratifications of the loci $\Sigma_k$, $k \geq 2$, as well. Prior to Theorem \ref{th.traversally_generic}, we will revisit this complication.\smallskip

We associate several  flags $\{\mathsf{Flag}_b(v)\}_{b \in \a^{-1}(a)}$ with each point $a \in \a(\d X)$. Let $\a^\d_\ast$ denote the differential of the immersion $\a^\d: \d X \to \hat X$. 

\begin{definition} We say that several vector subspaces $\{V_i \subset W\}_i$ of a given vector space $W$ are {\sf in general position}, if the obvious map $W \to \oplus_i (W/V_i)$ is onto. Note that this definition allows for any numbers of $V_i$'s to coincide with the ambient $W$.\smallskip

We say that the flags $\{(\a^\d_\ast)[\mathsf{Flag}_b(v)]\}_{b \in \a^{-1}(a)}$  are  {\sf in general position} in the ambient space $T_a\hat X$, if the $\a_\ast$-images of the minimal strata $\{F^{n-j(b)+1}(v)\}_{b \in \a^{-1}(a)}$ of the flags $\{\mathsf{Flag}_b(v)\}_{b \in \a^{-1}(a)}$ are {\sf in general position} in $T_a\hat X$. \hfill $\diamondsuit$
\end{definition}

\begin{example} \emph{Consider the case $n=2$, depicted in Fig.\ref{fig.A_generic}.}

\emph{If $\#(\a^\d)^{-1}(a) =1$,  then each flag $(\a^\d_\ast)\{F^2 \supset F^1 \supset F^0\}$, or $(\a^\d_\ast)\{F^2 \supset F^1\}$, or $(\a^\d_\ast)\{F^2\}$ is in general position at $a$. } \smallskip

\emph{If $\#(\a^\d)^{-1}(a) =2$,  a pair of flags is in general position,  if and only if, it is of the form $(\a^\d_\ast)\{F^2 \supset F^1 \}$ and $(\a^\d_\ast)\{G^2 \}$, or of the form $(\a^\d_\ast)\{F^2 \}$ and $(\a^\d_\ast)\{G^2 \}$.} \smallskip

\emph{If $\#(\a^\d)^{-1}(a) =3$,  a triple of flags is in general position only if the pair is of the form $(\a^\d_\ast)\{F^2 \}$, $(\a^\d_\ast)\{G^2 \}$, $(\a^\d_\ast)\{H^2 \}$. The rest of combinations fail to be generic.} \hfill $\diamondsuit$
\end{example}

\begin{figure}[ht]
\centerline{\includegraphics[height=2.in,width=4.3in]{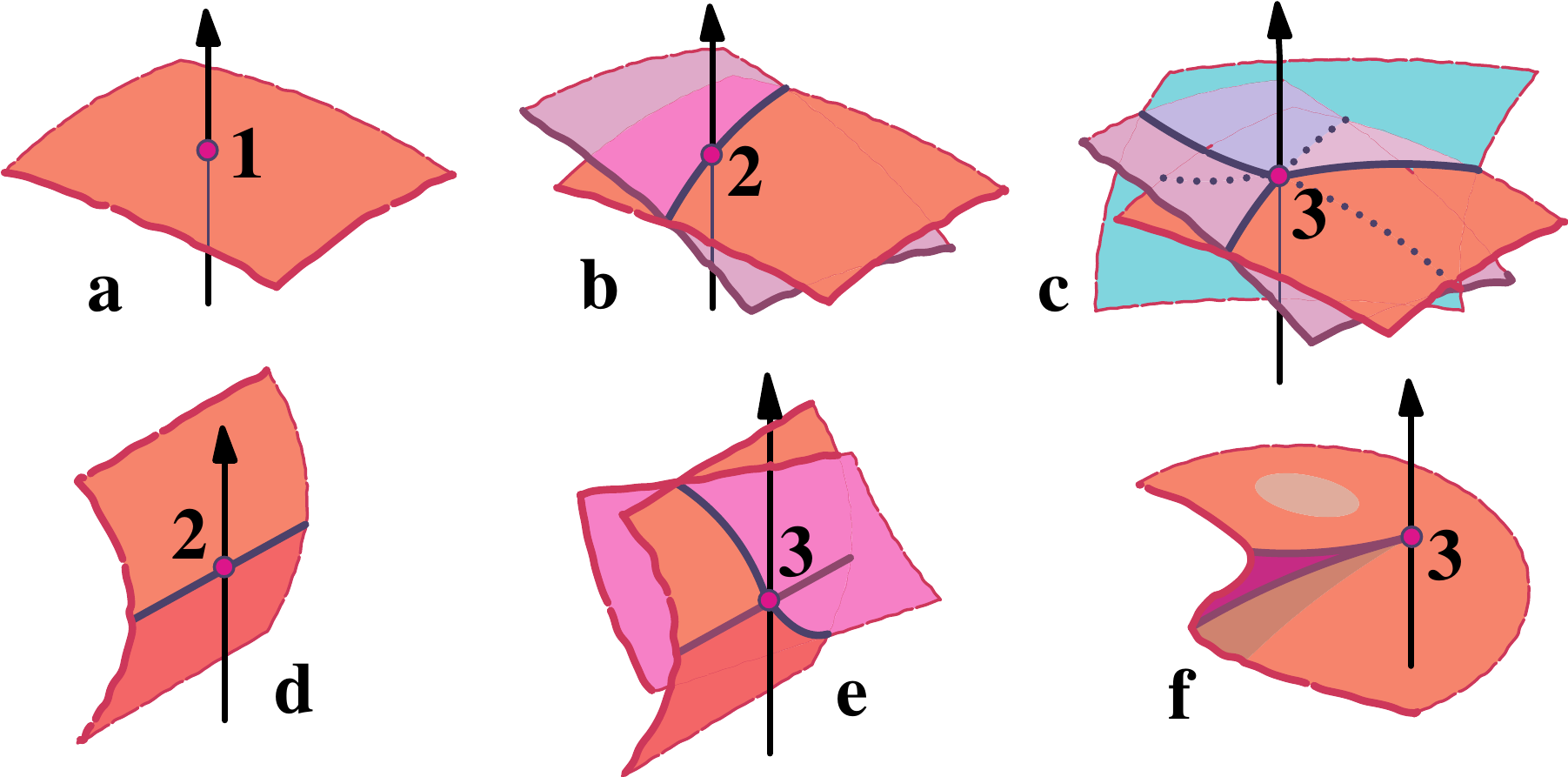}}
\bigskip
\caption{\small{Six locally generic configurations of $\a^\d(\d X)$ in 3D. The numbers $1,2,3$ reflect the local multiplicity of the marked point $a$ on the trajectory $\hat\g_a$ through it.}} 
\label{fig.A_generic}
\end{figure}

\begin{definition}\label{def.multiplicity_of_sub} Let $X, \hat X$ be smooth compact $(n+1)$-manifolds with boundary and $\hat v$ a traversing vector field on $\hat X$. We assume that a submersion $\a: X \to \hat X$ has the following properties: for each point $a \in \a(\d X)$, there exist a natural number $k = k(a) \leq n+1$, an open neighborhood $U_a$ of $a$ in $\hat X$, and smooth functions $\{z_1, \dots , z_k:  U_a \to \R\}$ such that: 
\begin{enumerate} 
\item $0$ is a regular value for each $z_i$, 
\item in $U_a$, the locus  $\a(\d X)$ is given by the equation $\{z_1\cdot  \ldots \cdot z_k = 0\}$,
\item the differential $k$-form $dz_1 \wedge \ldots \wedge dz_k \, |_{U_a} \neq 0.$\footnote{Thus $\a^\d$ is $k$-{\sf normal} in the sense of Definition 3.3 from \cite{K9}.} 
\end{enumerate}
Let $\hat\g_a$  denote the $\hat v$-trajectory through $a$.\smallskip

We say that a point $a \in \a(\d X)$ has a {\sf multiplicity} $j = j(a)$ with respect to $\hat v$, if the jet $\mathsf{jet}^{j-1}_a\big(z_1\cdot  \ldots \cdot z_k\big |_{\hat\g_a}\big) = 0$, but $\mathsf{jet}^{j}_a\big(z_1\cdot  \ldots \cdot z_k\big |_{\hat\g_a}\big) \neq 0$. 
  \hfill $\diamondsuit$
\end{definition}



\begin{definition}\label{def.boundary_generic} Let $X, \hat X$ be smooth compact $(n+1)$-manifolds with boundary and $\hat v$ a traversing vector field on $\hat X$. Let $U$ be an open $\hat v$-flow adjusted neighborhood of $a$ in $\hat X$. For each point $a \in \mathsf{int}(\hat X)$, consider a smooth transversal section $S$ of the $\hat v$-flow at $a$ and the flow-generated local projections $\pi: U \to S$. 

We say that a submersion $\a: X \to \mathsf{int}(\hat X)$  is {\sf locally generic} relative to $\hat v$ if, for each point $a \in \a(\d X)$, 
\begin{itemize}
\item the images of the flags\footnote{equivalently, of their minimal strata} $\{\mathsf{Flag}_b(v)\}_{b \in \a^{-1}(a)}$, under the differentials $(\a^\d)_\ast$, are in general position in $T_a(\hat X)$, 

\item the images of the flags $\{\mathsf{Flag}_b(v)\}_{b \in \a^{-1}(a)}$, under the differentials $(\pi \circ \a^\d)_\ast$, are in general position in the tangent space $T_aS$. \hfill $\diamondsuit$
\end{itemize}
\end{definition}

\smallskip

One may compare the next definition, which utilizes the notion of convexity, with Definition \ref{def.k-flat_envelops}, introducing the more general notion of  $k$-convexity.

\begin{definition}\label{def.convex} Let $\hat X$ be a compact connected smooth manifold with boundary, equipped with a vector field  $\hat v$. We say that the pair $(\hat X, \hat v)$ is {\sf  convex} if 
\begin{itemize}
\item $\hat v$ admits a Lyapunov function $\hat f: \hat X \to \R$  (i.e.,
$d\hat f(\hat v) > 0$); equivalently, $\hat v$ is {\sf traversing}, 

\item $\d\hat X$ is locally generic with respect to $\hat v$,
\item $\d_2^+\hat X(\hat v) = \emptyset$ (equivalently, $\hat v$ is {\sf convex}).
\hfill $\diamondsuit$
\end{itemize}
\end{definition}

\begin{remark}
\noindent
\emph{Any compact connected manifold $\hat X$ with boundary has a traversing vector field \cite{K1}. However, not any compact manifold with boundary admits a traversing} convex \emph{vector field! For example, consider any surface $\hat X$, obtained from a closed oriented connected surface, different from the $2$-sphere, by removing an open disk. Such an $\hat X$ does not admit a traversing convex vector field \cite{K5}.} \smallskip

\emph{By \cite{K1}, Lemma 4.1, any traversing vector field $\hat v$, admits a  Lyapunov function.}

\emph{If $Y$ is a \emph{closed} manifold, then any vector field $\hat v \neq 0$ that is tangent to the fibers of the obvious projection $[0, 1] \times Y \to Y$  is evidently convex with respect to $\d[0, 1] \times Y$. The obvious function $\hat f: [0, 1] \times Y \to [0, 1]$ has the desired Lyapunov property $d\hat f(\hat v) > 0$.} 
\hfill $\diamondsuit$
\end{remark}

\begin{example}\label{ex.non-trapping}
\emph{Any {\sf non-trapping} Riemannian metric $g$ on a compact smooth manifold $M$ with a {\sf convex} boundary $\d M$ produces the \emph{geodesic vector field} $\hat v^g$ on the unit spherical bundle $SM \to M$, which is traversing and convex with respect to $\d(SM)$ \cite{K6}. The very existence of such a metric $g$ puts severe restrictions on the topological nature of $M$.} \hfill $\diamondsuit$ 
\end{example}

Let $\mathsf{conv}(\hat X)$ denote the space of traversing vector fields $\hat v$ on $\hat X$ such that $(\hat X, \hat v)$ is a convex pair, as in Definition \ref{def.convex}. The space $\mathsf{conv}(\hat X)$ is considered in the $C^\infty$-topology.

\begin{lemma}\label{lem.conv_fields} 
For any  $\hat v_0, \hat v_1 \in \mathsf{conv}(\hat X)$ that belong to the same path-connected component of the space $\mathsf{conv}(\hat X)$, there exists a smooth diffeomorphism $\phi: \hat X \to \hat X$ such that $\phi$ maps $\hat v_0$-trajectories to $\hat v_1$-trajectories, while preserving their orientations.
\smallskip

If, for a pair $\hat v_0, \hat v_1 \in \mathsf{conv}(\hat X)$, there exists a smooth isotopy $\{\psi^t: \d\hat X \to \d\hat X\}_{t \in [0,1]}$ such that $\psi^1(\d_1^+\hat X(\hat v_0)) = \d_1^+\hat X(\hat v_1)$, then exists a smooth diffeomorphism $\phi: \hat X \to \hat X$, an extension of $\psi^1$,  that maps $\hat v_0$-trajectories to $\hat v_1$-trajectories, while preserving their orientations.
\end{lemma}

\begin{proof} If $\hat v \in \mathsf{conv}(\hat X)$ is a convex boundary generic vector field, then by Theorem 6.6 from \cite{K7}, the stratification $\hat X \supset \d_1^+\hat X(\hat v) \supset \d_2^-\hat X(\hat v)$ is stable, up to an isotopy of $\hat X$, under sufficiently small perturbations of $\hat v$. As a result, within a path connected component of $\hat v$ in $\mathsf{conv}(\hat X)$, the smooth topological type of this Morse stratification is stable via an isotopy. In particular, $\d_1^+\hat X(\hat v_0)$ and $\d_1^+\hat X(\hat v_1)$ are isotopic in $\hat X$, provided the $\hat v_0$ and $\hat v_1$ are connected by a path in $\mathsf{conv}(\hat X)$.
\smallskip

Let us denote by $\{\hat\psi^t: \hat X \to \hat X\}_{t \in [0,1]}$ the isotopy that transforms the $\hat v_0$-induced Morse stratification of $\d\hat X$ to the $\hat v_1$-induced Morse stratification of $\d\hat X$. 
Let us compare the vector fields $\tilde v_0 =_{\mathsf{def}} (\hat\psi^1)_\ast(\hat v_0)$ and $\hat v_1$.  Both vector fields point inside of $\hat X$ exactly along $\d_1^+\hat X(\hat v_1)$. Since $\hat v_0, \hat v_1$ are traversing, they admit some Lyapunov  functions $f_0, f_1: \hat X \to \R$. 
Put $\tilde f_0 =_{\mathsf{def}} ((\hat\psi^1)^{-1})^\ast(f_0)$. It serves as Lyapunov's function for $\tilde v_0$. 
For $x \in \d_1^+\hat X(\hat v_1)$, we denote by $\g_x^{\{0\}}$ and $\g_x^{\{1\}}$ the $\tilde v_0$- and $\hat v_1$-trajectories through $x$. Let $var_0(x)$ stands for the variation of the function $\tilde f_0$ along $\g_x^{\{0\}}$ and $var_1(x)$ for the variation of the function $\hat f_1$ along $\g_x^{\{1\}}$.

For $y \in \g_x^{\{0\}}$, consider the unique point $\tilde\phi(y) \in \g_x^{\{1\}}$ such that $f_1(\tilde\phi(y))/var_1(x) = \tilde f_0(y)/var_0(x)$.
Now the diffeomorphism $\phi := \tilde \phi \circ \hat\psi^1$ takes $\hat v_0$-trajectories to $\hat v_0$-trajectories, while preserving their orientations. Hence, we have shown that  the isotopy class of $\d_1^+\hat X(\hat v)$ in $\d X$ determines the smooth topological type of a convex pair $(\hat X, \mathcal L(\hat v))$, where $\mathcal L(\hat v)$ denotes the oriented $1$-foliation, determined by a convex traversing $\hat v$.
\end{proof}





\begin{lemma} Let $\hat X$ is a connected compact smooth manifold with boundary. If 
$\#(\pi_0(\d\hat X)) \geq 3$, then $\mathsf{conv}(\hat X) = \emptyset$. 
\end{lemma}

\begin{proof} For a convex $\hat v$, $\d_1^+ \hat X(\hat v)$ and $\d_1^- \hat X(\hat v)$, each is a deformation retract of $\hat X$. Therefore, each of these two loci must be connected. Thus each of the loci $\d_1^+ \hat X(\hat v)$, $\d_1^- \hat X(\hat v)$ must be contained in some connected component of $\d\hat X$. When $\#(\pi_0(\d\hat X)) \geq 3$, this argument forces at least one component of the boundary to be free from both loci. However, the union of the two loci is the entire boundary. This contradiction proves the claim.
\end{proof}

\subsection{Homology spheres and convex  flows}
Let $\mathbf{HS}_{n-1}$ denote the monoid of smooth types of {\sf integral homology $(n-1)$-spheres}, being considered up to connected sums with smooth {\sf homotopy} $(n-1)$-spheres. 

For $n \neq 4$, let $\mathbf \Theta_n$ denote the group of $h$-cobordism classes of {\sf smooth homotopy $n$-spheres}. The operations in $\mathbf{HS}_{n-1}$ and in $\mathbf \Theta_n$ are the connected sums of spheres. We denote the order of the group $\mathbf \Theta_n$ by $|\mathbf \Theta_n|$.

\begin{proposition}\label{prop.homology_sphere} Any convex vector field $\hat v$ on the standard ball $D^{n+1}$ defines a smooth involution $\tau_{\hat v}$ on $S^n = \d D^{n+1}$, whose fixed point set $\d_2^-D^{n+1}(\hat v)$ is an integral homology $(n-1)$-sphere. 

For $n \geq 6$, any element of the set $\mathbf{HS}_{n-1}$ arises as the locus $\d_2^-\hat X(\hat v)$ for a convex traversing vector field $\hat v$ on a smooth compact contractible $(n+1)$-dimensional manifold $\hat X$, whose boundary $\d \hat X$ is a smooth homotopy sphere. Moreover, $\d \hat X$ admits a smooth involution $\tau_{\hat v}$ such that $(\d \hat X)^{\tau_{\hat v}} = \d_2^-\hat X(\hat v)$.\smallskip

For $n \geq 6$, the multiple $|\mathbf \Theta_n| \cdot [\Sigma]$ of any given element $[\Sigma] \in \mathbf{HS}_{n-1}$ arises as the locus $\d_2^-D^{n+1}(\hat v)$ for a convex traversing vector field $\hat v$ on the ball $D^{n+1}$. The sphere $\d D^{n+1}$ admits a smooth involution $\tau_{\hat v}$ such that $(\d D^{n+1})^{\tau_{\hat v}} = \d_2^-D^{n+1}(\hat v)$.
\end{proposition}

\begin{proof} Using a convex $\hat v$-flow, $\d_1^+D^{n+1}(\hat v)$ is a deformation retract of $D^{n+1}$ and thus a contractible manifold. By Poincar\'{e} duality, $\d_2^-D^{n+1}(\hat v)$, the boundary of $\d_1^+D^{n+1}(\hat v)$, is a homology sphere.\smallskip

By \cite{Ke}, Theorem 3, for $n -1 \geq  5$, any smooth homology sphere $\Sigma^{n-1}$, after a connected sum $\Sigma^{n-1} \#\, \Sigma^{n-1}_H$ with a unique smooth homotopy $(n-1)$-sphere $\Sigma^{n-1}_H$, bounds a contractible smooth manifold $W^n$. \smallskip

Consider a smooth metric $g$ on $W^n$. We denote by $d_g(x, \d W^n)$ the smooth distance function to $\d W^n$ on a collar $U$ of $\d W^n$ in $W^n$. Let $F: W^n \to \R_+$ be a function that is strictly positive and smooth in the interior of $W^n$ and coincides with the function $\tilde F(x) := \sqrt{d_g(x, \d W^n)}$ in the collar $U$. 

Consider a smooth manifold $\hat X^{n+1} \subset \R \times W^n$, given by the inequality $\{(t, x): |t | \leq F(x)\}$. It comes with the vector field $\hat v$ that is tangent to the fibers of the obvious projection $p: \R \times W^n \to W^n$. Since $W^n$ is contractible, the boundary of $\hat X^{n+1}$, the double of $W^n$, is a smooth homotopy $n$-sphere, and $\hat X^{n+1}$ is a homotopy ball. 
The vector field $\hat v$ on $\hat X^{n+1}$ is convex and defines an involution $\tau_{\hat v}: \d \hat X^{n+1} \to \d \hat X^{n+1}$, whose fixed point set $(\d \hat X^{n+1})^{\tau_{\hat v}} = \d_2^-\hat X^{n+1}(\hat v) = \d W = \Sigma^{n-1} \#\, \Sigma^{n-1}_H$. This validates the second claim.

The connected sum of $|\mathbf \Theta_n|$ copies of $\d \hat X^{n+1}$ is a standard $n$-sphere. Consider the boundary connected sum $(\hat Y^{n+1}, \tilde v)$ of the $|\mathbf \Theta_n|$ copies of the pair $(\hat X^{n+1}, \hat v)$. Here the $1$-handles $H = D^n \times [0, 1]$ are attached at pairs of points that belong to different pairs of $\d W^n$'s so that $|\mathbf \Theta_n|$ copies of $W^n$ are connected by the chain of $1$-handles $D^{n-1}_+ \times [0, 1]$, where $D^{n-1}_+ \subset \d D^n$ is a hemisphere. The fields $\tilde v$ in the different copies extend concavely across the $1$-handles.  Then $\d \hat Y^{n+1}$ is the standard sphere $S^n$ which bounds a contractible manifold $\hat Y^{n+1}$. By the $h$-cobordism theorem (see \cite{Mi}), applied to $\hat Y^{n+1} \setminus D^{n+1}$, we conclude that $\hat Y^{n+1}$ is the standard ball.
Thus we managed to build a convex vector field $v^\#$ on $D^{n+1}$ whose locus  $\d_2^-D^{n+1}(v^\#)$ is a homology sphere $\tilde\Sigma^{n-1}$, the $|\mathbf \Theta_n|$-multiple of the given class $\Sigma^{n-1} \in \mathbf{HS}_{n-1}$. In particular, 
$\d D^{n+1}$ admits a smooth involution $\tau_{v^\#}$ whose fixed point set is $|\mathbf \Theta_n| \cdot \Sigma^{n-1}$.
\end{proof}

\begin{example}\label{ex.icosohedron} \emph{Consider a free action of the icosahedral group $\mathsf I_{120}$ on $S^3$. Then $\mathsf I_{120}$ acts freely on $S^7 = \mathsf{join}(S^3, S^3)$. Hence the orbit space $\Sigma^7 =_{\mathsf{def}} S^7/\mathsf I_{120}$ is a homology sphere. By Proposition \ref{prop.homology_sphere}, there is a convex traversing vector field $\hat v$ on the ball $D^9$, such that its locus $\d_2^-D^9(\hat v)$ is a connected sum of $28$ copies of $[\Sigma^7] \in \mathbf{HS}_7$, and its locus $\d_1^+D^9(\hat v)$ is contractible. Note that $\pi_1(\d_2^-D^9(\hat v))$ is a free product of 28 copies of $\mathsf I_{120}$.}
\hfill $\diamondsuit$
\end{example}

\begin{corollary}For $n > 6$, $\pi_0(\mathsf{conv}(D^{n+1}))$ admits a surjection onto a subgroup $\mathbf{G}_{n-1}$ of $\mathbf{HS}_{n-1}$ of index $|\mathbf \Theta_n|$ at most. 

In particular, here are a few ``clean" surjections: 
$$\pi_0(\mathsf{conv}(D^{7})) \to \mathbf{HS}_{5}, \; \pi_0(\mathsf{conv}(D^{13})) \to \mathbf{HS}_{11}, \;\pi_0(\mathsf{conv}(D^{62})) \to \mathbf{HS}_{60}.$$
\end{corollary}

\begin{proof} For $n-1 \geq 5$, let $\Sigma^{n-1}$ be a given smooth homology sphere, and let $\Sigma^{n-1}_H$ be the unique homotopy sphere such that $\Sigma^{n-1} \#\, \Sigma^{n-1}_H$ bounds a smooth contractible manifold \cite{Ke}. By Proposition \ref{prop.homology_sphere}, any convex traversing vector field $\hat v$ produces a homology sphere $\d_2^-D^{n+1}(\hat v)$, and the $|\mathbf \Theta_n|$-multiple of any $\Sigma^{n-1} \#\, \Sigma^{n-1}_H$ is produced this way.  On the other hand, deforming $\hat v$ within the space of boundary generic vector fields does not change the smooth isotopy type of the pair $\d_1^+D^{n+1}(\hat v) \supset  \d_2^+D^{n+1}(\hat v)$ (see Lemma \ref{lem.conv_fields}, or \cite{K1}, \cite{K2}). In particular, the smooth topological type of the pair is preserved along a path in the space $\mathsf{conv}(D^{n+1})$. Therefore, $\pi_0(\mathsf{conv}(D^{n+1}))$ admits a surjection onto a subgroup $\mathbf{G}_{n-1}$ of $\mathbf{HS}_{n-1}$ of index $|\mathbf \Theta_n|$ at most. \smallskip

The three examples of surjections in the corollary are based on the computations of $|\mathbf \Theta_n|$ in \cite{KeM}, \cite{WaX}; we just use some $n$'s for which $|\mathbf \Theta_n| =1$.
\end{proof}

\subsection{Convex pseudo-envelops of traversing  flows}

The next key lemma incapsulates a given convex traversing flow $(\hat X, \hat v)$ into the obvious traversing flow $\tilde v$ in a {\sf box} $[0, 1] \times Y$ for an appropriate choice of a compact manifold $Y$, $\dim Y = \dim \d\hat X$. The construction that realizes the embedding  $(\hat X, \hat v) \subset ([0, 1] \times Y, \tilde v)$ delivers a global ``virtual section" $\{0\} \times Y$ of $\hat v$. In turn, this section enables us to apply the key Theorem 3.1 from \cite{K9} to any convex traversing flow $(\hat X, \hat v)$. Therefore, we will be able to transfer many results from \cite{K9} and from Section 3 to the environment of convex envelops and  pseudo-envelops (see Definition \ref{def.convex}) of traversing boundary generic vector fields.  

\begin{lemma}\label{lem.box} Let $\dim(\hat X) = n+1$. If a pair $(\hat X, \hat v)$ is convex, then there exists a compact smooth $n$-manifold $Y$ such that: 
\begin{enumerate}
\item $\hat X \subset [0, 1] \times Y$,
\item $\hat v$ is tangent to the fibers of the projection $p: [0, 1] \times Y \to Y$,
\item the obvious function $h: [0, 1] \times Y  \to [0, 1]$ has the property $d h(\hat v) > 0$, 
\item with the help of $p$, the loci $\d_1^+\hat X(\hat v)$ and $\d_1^-\hat X(\hat v)$  each is homeomorphic to $Y$.
\end{enumerate}
\end{lemma}

\begin{proof} 

Since $v$ is a traversing field, it admits a Lyapunov function $\hat f: \hat X \to \R$ so that $d\hat f(v) > 0$ in $\hat X$ \cite{K1}. We add a collar $U$ to $\hat X$ along its boundary $\d\hat X$ and denote $\hat X \cup_{\d\hat X} U$ by $\tilde X$. Then, we smoothly extend $\hat v$ and $\hat f$ in $\tilde X$ and denote these extensions by $\tilde v$ and $\tilde f$. We adjust $U$ so that $\tilde v \neq 0$ and $\tilde f(\tilde v) >0$ there.

Let $F: \tilde X \to \R$ be a smooth function such that $0$ is its regular value and $F^{-1}(0) = \d\hat X$, $F^{-1}((-\infty, 0]) = \hat X$. Let $\tilde X(F, \e) =_{\mathsf{def}} \{ x \in \tilde X|\; F(x) \leq \e\}$. For a sufficiently small $\e > 0$, $\tilde X(F, \e)$ is a smooth compact manifold, contained in $\tilde X$.

By definition, $\d_2\hat X(\tilde v) := \{x \in \d\hat X |\; F(x) = 0, \text{ and } (\mathcal L_{\tilde v} F)(x) = 0\}$. The convexity of $\hat v$ in $\hat X$ means that $\d_2^+\hat X(\hat v) = \emptyset$, thus $\d_2^-\hat X(\hat v) = \d_2\hat X(\hat v)$.
By Morin's Theorem \cite{Mor} (see \cite{K2} for details), in the vicinity of each point $x \in \d_2^-\hat X(\tilde v)$ there is a system of smooth coordinates $(u, w, y_1, \ldots , y_{n-1})$ in which $F((u, w, y_1, \ldots , y_{n-1})) = u^2 +w$, so that $\d\hat X$ is given by the equation $u^2 +w = 0$, $\hat X$  by the inequality $u^2 +w \leq 0$, and each $\tilde v$-trajectory is produced by freezing the coordinates $w$ and $y_1, \ldots , y_{n-1}$. 

Since  $d\tilde f(\tilde v) > 0$ in $\tilde X(F, \e)$, the field $\tilde v$ is traversing in $\tilde X(F, \e)$.  Hence each $\tilde v$-trajectory $\g \subset \tilde X(F, \e)$ is either transversal to $\d\hat X$ at a pair of points, or it is simply tangent to $\d\hat X$ at a singleton, or does not intersect $\hat X$. 

Consider the set $Z \subset \tilde X(F, \e)$ of $\tilde v$-trajectories through the points of $\hat X$ (equivalently, the set of $\tilde v$-trajectories through $\d_1^+\hat X(\hat v)$). Every such trajectory $\g$ is a closed oriented segment $[a(\g), b(\g)]$, where $a(\g)\neq b(\g) \in \d Z \subset \d \tilde X(F, \e)$. Thus, the variation $var_\g(\tilde f)$ of the function $\tilde f$ along $\g$ is strictly positive. Using compactness of $\hat X$, we get that $\inf_{\g \subset Z}\;\{var_\g(\tilde f)\} > 0$. 

Let $Y =_{\mathsf{def}} \bigcup_{\g \subset Z} a(\g) \subset \d Z$. Using the local model $\{u^2 +w \leq \e\}$, we see that $\tilde v$ is transversal to $Y$ for all sufficiently small $\e> 0$. 

Now, let us consider a new function on $\g$: 
$$h_\g(x) = _{\mathsf{def}}(var_{[a(\g),\, x]} \tilde f) \big/ (var_{[a(\g),\, b(\g)]} \tilde f).$$ 
The  function $h: Z \to \R$, defined as a collection of functions $\{h_\g: \g \to \R\}_{\g \subset Z}$, is evidently a new Lyapunov function for $\tilde v$ on $Z$. It is smooth thanks to the transversality of $\tilde v$ to $Y$ and the smooth dependence of solutions of ODEs on initial data. In fact, $h$  gives the product structure $[0, 1] \times Y$ to $Z$. Indeed, any point $x \in Z$ is determined by the unique trajectory $\g_x$ through $x$ and by the value at $x$ of the Lyapunov function $h$ of the $\tilde v$-flow.

Finally, the $(-\tilde v)$-flow defines a smooth map $p: \d_1^+\hat X(\hat v) \to Y$ which is a homeomorphism. In fact, $\d_1^+\hat X(\hat v)$ is diffeomorphic to $Y$ by a small perturbation of $p$.
\end{proof}

Let  $\a: X \to \hat X$ be a smooth submersion of a compact smooth manifold $X$ with boundary into the interior of a compact connected smooth manifold $\hat X$ of the same dimension, where $\d\hat X \neq \emptyset$.  Assume that  $\hat X$ is equipped with a traversing vector field $\hat v$ so that the pair $(\hat X, \hat v)$ is {\sf  convex} in the sense of Definition \ref{def.convex}. As before, we denote by $v = \a^\dagger(\hat v)$ the pull-back of $\hat v$ under $\a$. We denote by $f$ the pull-back $\a^\ast(\hat f)$ of the Lyapunov function $\hat f$.
When $\a$ is a regular embedding, to simplify the notations, we identify $X$ and $\a(X)$. Under this identification, $v = \hat v|_X$ and $f = \hat f|_X$.


\begin{definition}\label{def.k-flat_envelops} Let $(\hat X, \hat v)$ be a convex pair and let $\a: X \to \hat X$ be a submersion for which $\a^\d$ is locally generic (in the sense of Definition \ref{def.boundary_generic}) with respect to $\hat v$. 

We say that $\a$ is $k$-{\sf convex} if $\d_k^+X(\a^\dagger(\hat v)) = \emptyset$,  $k$-{\sf concave} if  $\d_k^-X(\a^\dagger(\hat v))  = \emptyset$, and $k$-{\sf flat} if $\d_k X(\a^\dagger(\hat v))  = \emptyset$. 
\hfill $\diamondsuit$
\end{definition}

Let a regular immersion $\a^\d: \d X \to \hat X$ be locally generic relative to $\hat v$ in the sense of Definition \ref{def.boundary_generic}. For each $\hat v$-trajectory $\hat\g \subset \hat X$, we pick a point $a \in \hat\g \cap \a(\d X)$.  Such $a$ belongs to the intersection of $k = k(a) \in [1, n+1]$ local branches of  $\a^\d(\d X)$, where $k(a) = \#((\a^\d)^{-1}(a))$. In particular, as we remarked before, $a$ belongs to a unique collection of the $\a$-images of the Morse strata $\{\a(\d_{j(b)}X(v))\}_{b\, \in \a^{-1}(a)}$ with the \emph{maximal possible} indexes $j(b)$ (recall that  $v := \a^\dagger(\hat v)$ is the $\a$-transfer of $\hat v$). 
The tangent spaces of these strata are in general position at $a$, thanks to $\a^\d$ being locally generic. Recall that this setting includes the cases where some or all the strata $\{\a(\d_{j(b)}X(v))\}_{b\, \in \a^{-1}(a)}$ are $n$-dimensional, i.e., $j(b) =1$.

\begin{eqnarray}\label{eq.Morse_branched}
\text{ Let }\; j(a) =_{\mathsf{def}} \sum_{b\, \in \a^{-1}(a)} j(b)\; \geq \; k(a) \text{ \; and} \nonumber \\
\d_{j(a)}\big(\a(X)\big)(\hat v) =_{\mathsf{def}} \bigcap_{\, b\, \in \a^{-1}(a)} \a\big(\d_{j(b)}X(v)\big),
\end{eqnarray} 
the latter equality being understood as an identity of the two germs at $a$ of the LHS and RHS loci.
By Definition \ref{def.boundary_generic}, the germ of  $\d_{j(a)}(\a(X))(\hat v)$ at $a$ is a smooth submanifold of $\hat X$, transversal to the trajectory $\hat\g_a$.\smallskip

Let $\mathsf T_a \subset T_a\hat X$ be the tangent space to the minimal stratum $\d_{j(a)}(\a(X))$ at $a$ (see (\ref{eq.Morse_branched})).
With the help of the $\hat v$-flow, the subspace $\mathsf T_a \subset T_a\hat X$ spreads to form a $(\dim(\hat X) -j(a))$-dimensional subbundle $\mathsf T^\bullet_a$ of the tangent bundle $T\hat X\big |_{\hat\g}$ along the trajectory $\hat\g$. We denote by $\mathsf T^\clubsuit_a$ the image of $\mathsf T^\bullet_a$ under the quotient map $T\hat X|_{\hat\g} \to T\hat X|_{\hat\g}\big/T\hat\g$, where $T\hat\g$ stands for the $1$-bundle, tangent to $\hat\g$. \smallskip


We introduce a sightly modified version of Definition 3.2 from \cite{K2}, a modification that applies to submersions $\a: X \to \hat X$.

\begin{definition}\label{def.traversally_generic} Let $\hat v$  be a traversing vector field on a compact connected smooth manifold on $\hat X$ with boundary. 
We say that a submersion $\a: X \to \textup{int}(\hat X)$  is {\sf traversally generic relative to} $\hat v$, if:
\begin{itemize}
\item   $\a^\d$ is locally generic in the sense of Definition \ref{def.boundary_generic} with respect to $\hat v$,
\item  for each $\hat v$-trajectory $\hat\g \subset \hat X$, the subbundles $\{\mathsf T^\clubsuit_a\}_{a\, \in \, \hat \g\, \cap \, \a(\d X)}$ are in general position in the normal to $\hat\g$ (trivial) $n$-bundle $(T\hat X|_{\hat\g})\big/T\hat\g$.  \hfill $\diamondsuit$
\end{itemize}
\end{definition}

\begin{example}\label{ex.traversally_gen} \emph{The patterns in Fig.\ref{fig.A_generic} may be stacked vertically along a trajectory $\hat\g$. To get a traversally generic piles in the vicinity of $\g$, we obey the following rules: (1) to any stack, we may add any number of configuration of type $a$ from Fig.1, as long as the prescribed parity of $m(\hat\g) \equiv 0 \mod 2$ is not violated; (2) no two configurations of multiplicity $3$ (of the types $c, e, f$) reside on $\hat\g$; (3) at most two configurations of multiplicity $2$ (of the types $b, d$) reside on $\hat\g$,  moreover, the $(-\hat v)$-directed projections on the transversal section $S$ to $\hat\g$ of the fold loci (as in $d$) or/and of the simple self-intersections (as in $b$) must be transversal in $S$. } \hfill $\diamondsuit$
\end{example}
 \smallskip

Let $\a$ be traversally generic relative to $\hat v$. Then, for any $\hat v$-trajectory $\hat\g$, by counting the dimensions of the bundles $\{\mathsf T^\clubsuit_a\}_{a\, \in\, \hat\g\, \cap\, \a(\d X)}$ and using that they are in general position in the $n$-dimensional bundle, normal to $\hat\g$, we get that the reduced multiplicity 
\begin{eqnarray}\label{eq.reduced_mult}
m'(\hat\g) =_{\mathsf{def}}  \sum_{a\, \in\, \hat\g\, \cap\, \a(\d X)} \big(j(a) -1\big)\, \leq \, n.
\end{eqnarray}
  For a traversally generic $\a$ and a trajectory $\hat \g$, let $\g \subset X$ be any segment in $\a^{-1}(\hat\g) \subset X$ which is bounded by a pair of points in $\d X$. Then, by Theorem 3.5 from \cite{K2}, the total multiplicity $m(\g) \leq 2n +2$.  Although, for a traversally generic $\a$, there is no $\a$-independent constraint on the total multiplicity $m(\hat\g) =_{\mathsf{def}} \sum_{x\, \in\, \hat\g\, \cap\, \a(\d X)} j(x)$, there is an universal  constraint on the cardinality of the subset $\hat\g_{\{\geq 2\}} \subset \hat\g \cap \a(\d X)$, consisting of points $x$ whose multiplicity $m(x) \geq  2$. Namely, $\#(\hat\g_{\{\geq 2\}}) \leq n$, since no more than $n$ proper vector subspaces may be in general position in an ambient  $n$-dimensional vector space. 
\smallskip

Now we are ready to introduce the central notion of a {\sf convex pseudo-envelop}. 

\begin{definition}\label{def.convex_envelop} Let $\hat v$ be a traversing vector field on a compact connected smooth manifold $\hat X$ with boundary. We assume that $(\hat X, \hat v)$ is convex in the sense of Definition \ref{def.convex}.

We call such a pair $(\hat X, \hat v)$ a {\sf convex pseudo-envelop} of a submersion $\a: X \to \textup{int}(\hat X)$, if $\a^\d$ is locally generic relative to $\hat v$. We think of $X$ as being equipped with the pull-back vector field $v = \a^\dagger(\hat v)$, so that $(X, v)$ is ``enveloped" by $(\hat X, \hat v)$.

If  $\a$ is a locally generic \emph{regular embedding}, then we call $(\hat X, \hat v)$ a {\sf convex envelop} of $\a$.

\hfill $\diamondsuit$
\end{definition}
\begin{remark} 
\noindent 
\emph{Not all traversally generic pairs $(X, v)$ admit convex envelops $(\hat X, \hat v)$. For example, if $X$ has two closed submanifolds (or singular cycles),  $M$ and $N$, of complementary dimensions with a nonzero algebraic intersection number $M \circ N$, then no convex envelop $(\hat X, \hat v)$ of $(X, v)$ exists. Indeed,  with the help of the $(-\hat v)$-flow, $M$ is cobordant in $\hat X$ to a cycle $M'$ which resides in $\d_1^+\hat X(\hat v)$.  Similarly, with the help of the $\hat v$-flow, $N$ is cobordant in $\hat X$ to a cycle $N'$ which resides in $\d_1^-\hat X(\hat v)$. Note that if $M \circ N \neq 0$ in $X$, then evidently the same property holds in any $\hat X \supset X$. Since $M'$ and $N'$ are disjoint cycles, their intersection number $M'  \circ N' = 0$, which contradicts to the assumption $M \circ N \neq 0$.}
\emph{In particular, the non-triviality of $\mathsf {sign}(X)$, the Wall relative signature of $X$ \cite{W}, obstructs the existence of a convex envelop for any traversing $v$ on $X$.}
\emph{As Fig.\ref{fig.A1} testifies, this argument does not rule out the existence of a convex \emph{pseudo}-envelop for $(X, v)$. In fact, Fig.\ref{fig.A1} shows that any compact oriented surface $X$ can be enveloped.} \hfill $\diamondsuit$
\end{remark}
We denote by $\mathsf{Sub}(X, \hat X)$ the space of smooth submersions $\a: X \to \mathsf{int}(\hat X)$. 
\smallskip

When $\a$ is an embedding, by Theorem 3.5 from \cite{K2}, in the space $\mathcal V_{\mathsf{trav}}(\hat X, \a)$ of traversing vector fields $\hat v$ on $\hat X$, 
there is an open and dense subset $\mathcal V^\ddagger(\hat X, \a)$ such that $\a$ is traversally generic (see Definition \ref{def.traversally_generic}) with respect to  $\hat v \in \mathcal V^\ddagger(\hat X, \a)$. \smallskip

On the other hand, the property of a submersion $\a$ to be traversally generic with respect to a \emph{given} traversing vector field $\hat v$  (see Definition \ref{def.traversally_generic}) is  an ``open" property in the $C^\infty$-topology on $\mathsf{Sub}(X, \hat X)$, 
since it may be expressed in terms of mutual transversality of the relevant strata in the appropriate jet spaces. 

We conjecture that the transversal generality of submersions $\a$ with respect to a fixed convex pair $(\hat X, \hat v)$ is also a ``dense" property. 
Among other things, the next Theorem \ref{th.traversally_generic} shows that this conjecture is valid for the regular \emph{embeddings} $\a$ which admit convex envelops. 
However, for general submersions $\a$, we are able only to prove that a somewhat \emph{weaker} property ``is dense". That property is described in the second claim of Theorem \ref{th.traversally_generic}. Speaking informally,  we can insure by $\a$-perturbations the general positions of the singularities of the maps $\a^\d: \d X \to \hat X$ and  $\pi \circ \a^\d: \d X \to \d_1^+\hat X(\hat v)$ separately, but not mutually. Here the map $\pi: \hat X \to \d_1^+\hat X(\hat v)$ is defined by the $(-\hat v)$-directed convex flow.


\begin{theorem}\label{th.traversally_generic} Let $(\hat X, \hat v)$ be a convex pair, and $X$ a compact smooth manifold with boundary, $\dim(X) = \dim(\hat X) = n+1$.   Assume that $\mathsf{Sub}(X, \hat X) \neq \emptyset$. 
\begin{itemize}
\item There is an open and dense subset $\mathcal O \subset \mathsf{Sub}(X, \hat X)$  such that, for any $\a \in \mathcal O$, the local branches of $\a(\d X)$ are in general position in $T_a\hat X$ at every point $a \in \a(\d X)$.
\smallskip

\item For any $\a \in \mathcal O$ and for each $\hat v$-trajectory $\hat\g \subset \hat X$, the $\hat v$-invariant subbundles $\{\mathsf T^\clubsuit_a\}_{a\, \in\; \hat\g\, \cap\, \a(\d X)}$, 
generated by the intersections 
$$\bigcap_{b \in (\a^\d)^{-1}(a),\; j(b) \geq 2} \a_\ast\big [T_b\big(\d_{j(b)}X(v)\big)\big]$$ of the tangent spaces to the Morse strata $\big\{\d_{j(b)}X(v)\big\}_{b \in (\a^\d)^{-1}(a),\; j(b) \geq 2}$, 
are in general position in the normal to $\hat\g$ (trivial) $n$-bundle $(T\hat X|_{\hat\g})\big/T\hat\g$. Here  $v$ is the pull-back of $\hat v$ under $\a$.
%

\end{itemize}
\end{theorem}

\begin{proof} We assume that the openness of $\mathcal O$ in $\mathsf{Sub}(X, \hat X)$ is clear, due to the compactness of $X$, and will present the arguments that validate the density of $\mathcal O$ in $\mathsf{Sub}(X, \hat X)$. We divide the proof into three steps, marked as {\bf (i)}, {\bf (ii)}, and {\bf (iii)}.\smallskip 


{\bf (i)} By \cite{LS}, the set $\mathcal N \subset C^\infty(\d X, \hat X)$ of smooth maps $\b: \d X \to \hat X$, such that  $\b$ is a $k$-{\sf normal} immersion in the sense of \cite{LS} (see also  Definition 3.3 from \cite{K9}) for all $k \leq n+1$, is open and dense. 


First we aim to show that, for a given submersion $\a \in \mathsf{Sub}(X, \hat X)$, there is an open set $\mathcal O_\a \subset \mathsf{Sub}(X, \hat X)$  such that $\a \in \mathsf{closure}(\mathcal O_\a)$ and, for any $\tilde\a  \in \mathcal O_\a$, the submersion  $\tilde\a^\d \in \mathcal N$.

With this goal in mind, we choose an auxiliary metric $g$ on $\hat X$ such that: $(1)$ the boundary $\d\hat X$ is convex in $g$,  and $(2)$ there is $\e_0 > 0$ such that any two points $x, y \in \hat X$ that are less than $\e_0$-apart are connected by a single geodesic arc. Using the submersion $\a$, we pull-back $g$ to a Riemannian metric $g^\dagger_\a$ on $X$.

For some $\e < \e_0$, the $\e$-neighborhood $C_\e \subset X$ of $\d X$ in the metric $g^\dagger_\a$ has a product structure $\psi: \d X \times [0, \e] \stackrel{\approx}{\to} C_\e$, so that the curve $\delta_x =_{\mathsf{def}} \psi(x \times [0, \e])$ is the unique geodesic in $g^\dagger_\a$, normal to $\d X$ at $x \in \d X$. Using the diffeomorphism $\psi$, we construct a smooth diffeomotopy $\{\psi_t: X  \to X\}_{t \in [0, 1]}$ (via the flow inward normal to $\d X$) so that $\psi_0 = \mathsf{id}_X$ and $\psi_1(X) = X \setminus \mathsf{int}(C_\e)$. 

Using the convexity of $\d\hat X$ in $g$ and the choice of $\e < \e_0$ we get the following claim: for any $\b$ that approximates $\a^\d$ and each $x \in \d X$, there exists the unique geodesic  $\hat\delta_x \subset \hat X$ in the metric $g$ that connects the points $\b(x)$ and  $\a \circ \psi_1(x)$. 

We pick such an approximation $\b \in \mathcal N$ of $\a^\d$. Let  a smooth map $\a^\sharp(\b, \e): X \to \hat X$ be defined by the two properties:  {\bf (1)} $\a^\sharp(\b, \e)|_{X \setminus C_\e} := \a \circ \psi_1$, {\bf (2)}  for each $x \in \d X$, the diffeomorphism $\a^\sharp(\b, \e)| : \delta_x \to \hat \delta_x$ is 
an isometry with respect to $g^\dagger_\a$ and $g$ along the two geodesic arcs. By picking $\e$ small enough and $\b \in \mathcal N$ sufficiently $C^\infty$-close to $\a^\d$, we get that $\a^\sharp(\b, \e)$ is  $C^\infty$-close to $\a$. Therefore, we may assume that $\a^\sharp(\b, \e) \in \mathsf{Sub}(X, \hat X)$ for an appropriate choice of $\b$ that approximates $\a^\d$. By its construction, $(\a^\sharp(\b, \e))^\d \in \mathcal N$. 
\smallskip 


{\bf (ii)} For a given smooth map $\b:  \d X \to  \mathsf{int}(\d_1^+\hat X(\hat v))$, let $\b_\ast$ denote the differential of $\b$.
$$\text{ Let \;} \mathcal S^{(1)}(\b) =_{\mathsf{def}} \big\{x \in \d X \big |\; \mathsf{rk} (\b_\ast) \leq n-1\big\}.$$ If $S^{(1)}(\b)$ is a smooth manifold, then the locus $$\mathcal S^{(1, 1)}(\b) =_{\mathsf{def}} \big\{x \in \mathcal S^{(1)}(\b) \big|\; \mathsf{rk} (\b_\ast|_{\mathcal S^{(1)}(\b)}) \leq n-2\big\}$$ is well-defined. Continuing this way, the filtration of $\d X$ by the loci  $\big\{\mathcal S_{[k]}(\b) =_{\mathsf{def}}  \mathcal S^{\om}(\b)\big \}_k$, where $k \leq n+1$ and  $\om =(\underbrace{1, 1, \ldots, 1}_{k})$ is introduced. Applying  Boardman's Theorems 15.1-15.3, \cite{Bo}, the subset $\mathcal B$ of maps $\b \in C^\infty(\d X, \d_1^+\hat X(\hat v))$, for which all the strata $\{\mathcal S_{[k]} =_{\mathsf{def}} \mathcal S_{[k]}(\b)\}$ are smooth manifolds and all the maps $$\big\{\b |: \mathcal S_{[k]}^\circ =_{\mathsf{def}} \mathcal S_{[k]} \setminus \mathcal S_{[k+1]} \to \mathsf{int}(\d_1^+\hat X(\hat v))\big\}_k$$ are immersions, is open ($\d X$ is compact) and dense. Note that each $x \in \d X$ belongs to a unique pure stratum $\mathcal S_{[k(x)]}^\circ$. Moreover, using the Thom Multijet Transversality Theorem (see \cite{GG}, Theorem 4.13), the subspace of $\mathcal B$, formed maps $\b$ for which all the spaces $\{\b_\ast(T_x\mathcal S_{[k(x)]})\}_{x \in (\pi \circ \b)^{-1}(y)}$ are in general position in $T_y\d_1^+\hat X(\hat v)$ for all $k(x) \geq 2$ and all $y \in  \mathsf{int}(\d_1^+\hat X(\hat v))$, form an open and dense subset $\mathcal B_{\mathsf{NC}} \subset \mathcal B$ (see \cite{GG}, Theorem 5.2).  Here ``$\mathsf{NC}$" abbreviates the condition known as ``{\sf normal crossings}". Note that if $k(x) = 1$, then $\b_\ast(T_xS_{[k(x)]})$ coincides with $T_{\b(x)}\d_1^+\hat X(\hat v)$. 
\smallskip

We stress that, applying the previous arguments to the map $\b = \pi \circ\a^\d$, the strata $\bigcap_{\, b\, \in \a^{-1}(a)} \a\big(\d_{j(b)}X(v)\big)$ that involve $b$'s with $j(b) =1$ become  ``$\b$-invisible", thanks to the ``erasing" action of $\pi_\ast$ on all $n$-dimensional spaces $\a\big(T_a(\d_{1}X(v))\big)$. Because of this short comming, we cannot claim that the traversally generic $\a$ to $\hat v$ (see Definition \ref{def.traversally_generic}) are dense in $\mathsf{Sub}(X, \hat X)$.
\smallskip 

{\bf (iii)}  Note that the map $\pi: \hat X \to \mathsf{int}(\d_1^+\hat X(\hat v))$, defined by the $(-\hat v)$-flow, is smooth due to $\hat v$ being  convex; moreover, its restriction to $\mathsf{int}( \hat X)$ is a submersion. 

For a given $\a \in \mathsf{Sub}(X, \hat X)$, we form the composition $\a^\flat =_{\mathsf{def}} \pi\circ \a^\d$. 

By {\bf (i)}, we can approximate $\a$ by a new submersion $\a_1 = \a^\sharp(\b, \e)$ such that $\a_1^\d \in \mathcal N$. 

By {\bf (ii)}, we can approximate $\a_1^\flat =_{\mathsf{def}} \pi \circ \a_1^\d$ by a smooth map $\b_1^\flat: \d X \to  \mathsf{int}(\d_1^+\hat X(\hat v))$ such that $\b_1^\flat \in  \mathcal B_{\mathsf{NC}}$.

Let us fix a Lyapunov function $\hat f: \hat X \to \R$ for $\hat v$. In the spirit of Lemma \ref{lem.box},  the Lyapunov function $\hat f$ for $\hat v$ and the projection $\pi: \hat X \to \d_1^+\hat X(\hat v)$ define global smooth ``coordinates" in the interior of $\hat X$. That is, each pair $(t, y)$, where $y \in  \d_1^+\hat X(\hat v)$ and $t \in \hat f(\pi^{-1}(y))$, determines a unique point $x \in \hat X$ such that $\hat f(x) =t$ and  $\pi(x) = y$.  Let $f_1 = \a^\ast_1(\hat f)$. \smallskip

For the map $\b_1^\flat \in \mathcal B_{\mathsf{NC}}$, we define a smooth map $\b_1 =_{\mathsf{def}} \b_1(\b_1^\flat, \hat f): \d X \to \hat X$ by the formula $\b_1(x) = y$, where $x \in \d X$ and $y$ is the unique point on the $\hat v$-trajectory $\hat \g$ through $\b_1^\flat(x) \in \d_1^+\hat X(\hat v)$ such that $\hat f(y) =  f_1(x)$. Note that by choosing $\b_1^\flat$ sufficiently close to $\a_1^\flat$, we insure that  $\b_1$ is sufficiently close to $\a_1^\d$. Thus we may assume that $\b_1 \in \mathcal N$.

Recycling the argument that revolves around the construction of $\a^\sharp(\b, \e): X \to \hat X$ in part {\bf (i)} of the proof, we form the submersion $\a_2 =_{\mathsf{def}} \a_1^\sharp(\b_1, \e): X \to \hat X$. By the construction of $\a_2$, we have $\a_2^\flat = \b_1^\flat \in \mathcal B_{\mathsf{NC}}$ and $\b_1 \in \mathcal N$.

Therefore, we have shown that any given submersion $\a$ admits an approximation by some $\a_2 \in \mathsf{Sub}(X, \hat X)$ such that $\a_2^\d \in \mathcal N$ and $\a_2^\flat \in \mathcal B_{\mathsf{NC}}$.  These are exactly the two properties that describe the space $\mathcal O$ in the theorem.  
\end{proof}

\begin{corollary}\label{cor.traversally_generic_for_emb} Let $(\hat X, \hat v)$ be a convex pair, and $X$ a compact smooth manifold with boundary, $\dim(X) = \dim(\hat X) = n+1$.  

The regular \emph{embeddings} $\a: X \subset \hat X$ that are traversally generic with respect to $\hat v$ (see Definition \ref{def.traversally_generic}) form an open and dense set in the space of all regular smooth embeddings. 
\end{corollary}

\begin{proof} Since, for a regular embedding $\a: X \hookrightarrow \hat X$, $\a^\d$ is an embedding, the first claim of Theorem \ref{th.traversally_generic} is vacuous, and the second claim insures that $\a$ is traversally generic. 
\end{proof}

\begin{remark}\label{rem.Sigma}\emph{Consider the $k$-multiple self-intersection manifolds  $\{\Sigma_k^{\a^\d}\}_k$ of $k$-normal (see \cite{LS}) immersions $\a^\d: \d X \to \hat X$. By definition, $\Sigma_k^{\a^\d}$ is a submanifold of the $k$-fold product $(\d X)^k$, the preimage of the diagonal  $\Delta \subset (\hat X)^k$ under the transversal to it map $(\a^\d)^k$. The projection $p_1$ of $\Sigma_k^{\a^\d}$ on the first factor $\d X$ of the product $(\d X)^k$ is an immersion \cite{LS}. By composing $p_1$ with $\a^\d$, we get an immersion of $\a^\d \circ p_1: \Sigma_k^{\a^\d} \to \hat X$. By using the convex $(-\hat v)$-flow, we get a map $\pi: \hat X \to \d_1^+\hat X(\hat v)$. Finally, we obtain a smooth composite map $\pi \circ \a^\d \circ p_1: \Sigma_k^{\a^\d} \to \d_1^+\hat X(\hat v)$.\smallskip} 

\emph{We notice that, under the hypotheses and notations of Theorem \ref{th.traversally_generic}, if  $\hat v$ is tangent at $a$ to the intersection $\bigcap_{b \in (\a^\d)^{-1}(a)}\; \a_\ast \big(\d_{j(b)}X(v)\big)$, then, evidently, each local branch $\a_\ast(\d_{j(b_\star)}X(v))$, $b_\star \in (\a^\d)^{-1}(a),$ is tangent to $\hat v$ at $a$. Therefore, $j(b_\star) \geq 2$. In other words, for $k \geq 2$, the singular locus of the map $\pi \circ \a^\d \circ p_1: \Sigma_k^{\a^\d} \to \d_1^+\hat X(\hat v)$ is always \emph{contained} in the singular locus of the map $\pi \circ \a^\d \circ p_1: (\d X)^k \to \d_1^+\hat X(\hat v)$.} 
\hfill $\diamondsuit$
\end{remark}

\subsection{Quasitopies of convex envelops and pseudo-envelops}

Now, let us modify Definition 3.7 from \cite{K9} and Definition \ref{quasi_isotopy_of_sub} from this paper, so that they apply to convex pseudo-envelops of traversing flows (see Fig. \ref{fig.A_quasitopy}). This modification  is central to our efforts. 

\begin{definition}\label{quasi_isotopy_of_fields}
Fix natural even numbers $d \leq d'$ and consider a closed subposets $\Theta' \subset \Theta$ of the universal poset $\mathbf\Om$ from Section 2, such that $(\emptyset) \notin \Theta$ . 

Let  $\hat X$ be a $(n+1)$-dimensional compact manifold and $\hat v$ a convex traversing vector field on it. 
Let $\hat Z =_{\mathsf{def}} \hat X \times [0, 1]$. We denote by $\hat v^\bullet$ the vector field on $\hat Z$ that is tangent to each slice $\hat X \times \{t\}$, $t \in [0, 1]$, and is equal to $\hat v$ there. 
\smallskip

We say that a two convex pseudo-envelops, $\a_0: X_0 \to \hat X$ and $\a_1: X_1 \to \hat X$, are $(d, d'; \mathbf c\Theta, \mathbf c\Theta')$-{\sf quasitopic} in $\hat X$, if there exists a compact smooth orientable $(n+2)$-manifold $W$,\footnote{with corners $\d X_0 \coprod \d X_1$} whose boundary $\d W = (X_0 \coprod X_1) \bigcup_{\{\d X_0 \coprod \d X_1\}} \delta W$, and a smooth submersion 
$A: W \to \hat Z$ so that:
\begin{itemize} 
\item $A|_{X_0} = \a_0$ and $A|_{X_1} = \a_1$; 
\smallskip  
  
\item for each $z \in \hat X \times\, \d[0, 1]$, the total multiplicity $m_A(\hat\g_z)$ of the $\hat v^\bullet$-trajectory $\hat\g_z$ through $z$, relatively to $A(\d X_0 \coprod \d X_1)$, satisfies the constraints  
$m_A(\hat\g_z) \leq d$, 
$m_A(\hat\g_z) \equiv 0 \mod 2$, and the combinatorial tangency pattern $\omega^A(\hat\g_z)$ of  $\hat\g_z$ with respect to $A(\d X_0 \coprod \d X_1)$ does not belong to $\Theta$;
\smallskip  

\item for each $z \in \hat Z$, the total multiplicity $m_A(\hat\g_z)$ of $\hat\g_z$ with respect to   $A(\delta W)$ satisfies the constraints 
$m_A(\hat\g_z) \leq d'$,  $m_A(\hat\g_z) \equiv 0 \mod 2$, and  the combinatorial tangency pattern $\omega^A(\hat\g_z)$ of $\hat\g_z$  with respect to $A(\delta W)$
does not belong to $\Theta'$;
\end{itemize}

 We denote by $\mathcal{QT}_{d, d'}^{\mathsf{sub}}(\hat X, \hat v;\, \mathbf{c}\Theta;  \mathbf{c}\Theta')$ the set of quasitopy classes of such convex pseudo-envelops $\a: (X, \a^\dagger(\hat v)) \to (\hat X, \hat v)$. \smallskip

 If we insist that $\a_0, \a_1$, and $A$ are embeddings, then we get $\mathcal{QT}_{d, d'}^{\mathsf{emb}}(\hat X, \hat v;\, \mathbf{c}\Theta;  \mathbf{c}\Theta')$, the set of quasitopy classes of convex envelops. \hfill $\diamondsuit$
\end{definition}

It is easy to check that the quasitopy of convex pseudo-envelops (convex envelops) \hfill\break $\a: (X, \a^\dagger(\hat v)) \to (\hat X, \hat v)$ is an equivalence relation. \smallskip

\begin{figure}[ht]
\centerline{\includegraphics[height=3.2in,width=4.3in]{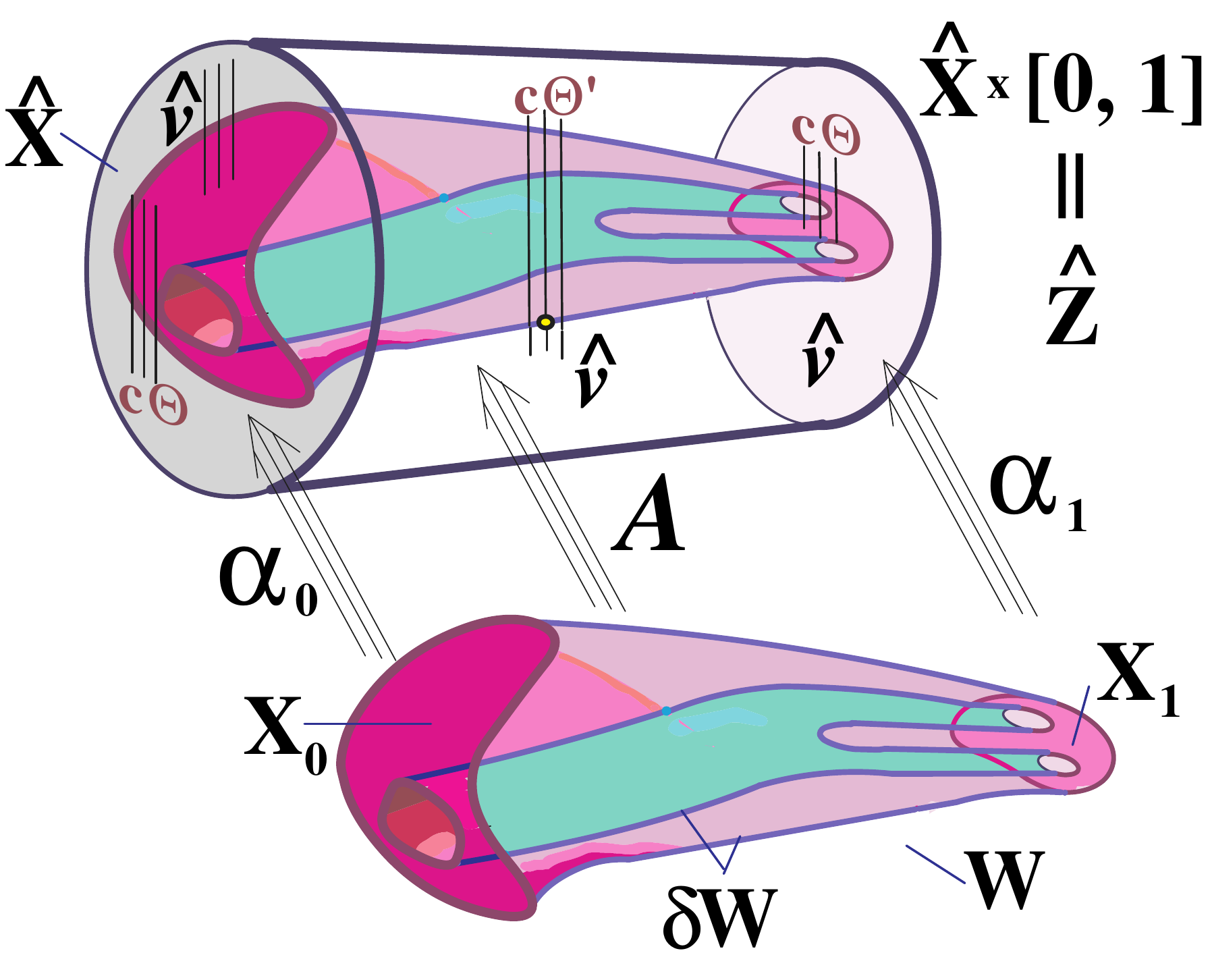}}
\bigskip
\caption{\small{The ingredients of Definition \ref{quasi_isotopy_of_fields}. For simplicity of the depiction, the maps $\a_0, \a_1, A$ are shown as embeddings.}} 
\label{fig.A_quasitopy}
\end{figure}

We are in position to state one of the main results of this paper.

\begin{theorem}\label{th.envelops}
Let $\Theta' \subset \Theta \subset \mathbf\Om$  be closed subposets that do not contain the element $(\emptyset)$.  Assume that $d' \geq d$ and $d' \equiv d \equiv 0 \mod 2$. Let $\hat X$ be a smooth compact connected $(n+1)$-dimensional manifold, equipped with a convex traversing vector field $\hat v$.\smallskip

Then, under the notations of Definition \ref{def.triples}, 
there is a canonical bijection
\begin{eqnarray}\label{eq.bijection, envelops}
\Phi^{\mathsf{emb}}: \mathcal{QT}^{\mathsf{emb}}_{d, d'}(\hat X, \hat v; \mathbf{c}\Theta; \mathbf{c}\Theta') \stackrel{\approx}{\longrightarrow} \nonumber \\
\stackrel{\approx}{\longrightarrow} \big[\big[(\d_1^+\hat X(\hat v),\, \d_2^-\hat X(\hat v)),\; \e_{d, d'}: (\cP_d^{\mathbf c\Theta}, pt) \to (\cP_{d'}^{\mathbf c\Theta'}, pt')\big]\big]
\end{eqnarray}
and a canonical surjection
\begin{eqnarray}\label{eq.bijection, pseudo-envelops}
\Phi^{\mathsf{sub}}: \mathcal{QT}^{\mathsf{sub}}_{d, d'}(\hat X, \hat v; \mathbf{c}\Theta; \mathbf{c}\Theta') \stackrel{epi}{\longrightarrow} \nonumber \\
\stackrel{epi}{\longrightarrow} \big[\big[(\d_1^+\hat X(\hat v),\, \d_2^-\hat X(\hat v)),\; \e_{d, d'}: (\cP_d^{\mathbf c\Theta}, pt) \to (\cP_{d'}^{\mathbf c\Theta'}, pt')\big]\big].
\end{eqnarray}
 The map $\Phi^{\mathsf{sub}}$ admits a right inverse.
\end{theorem}

\begin{proof} By Lemma \ref{lem.box}, any convex pseudo-envelop  $\a: X \to \hat X$ may be incapsulated into a convex pseudo-envelop $\tilde\a: X \to \hat X \subset [0, 1] \times Y $ of the product type (we called a capsule $[0, 1] \times Y $  ``{\sf a box}"), where $Y$ is homeomorphic to $\d_1^+\hat X(\hat v)$, and its boundary $\d Y$ to $\d_2^-\hat X(\hat v)$. 

Conversely, any submersion $\tilde\a: X \to  [0, 1] \times Y$, where $Y$ is a smooth compact $n$-manifold, by rounding the corners $\d[0, 1] \times \d Y$ of $[0, 1] \times Y$ produces a convex pseudo-envelop $(\hat X, \hat v) \subset ([0, 1] \times Y, \tilde v)$, where $\tilde v \neq 0$ is tangent to the fibers of $[0, 1] \times Y \to Y$. Similarly, the cobordisms  $A$ between pairs of quasitopies $\a_0, \a_1$ can be incapsulated in boxes of the form $([0, 1] \times Y) \times [0, 1]$.  Therefore, all the results from Section 3 in \cite{K9} apply to $\tilde\a^\d$; in particular, the pivotal Theorem 3.2 from \cite{K9} 
applies. With its help, the maps $\Phi^{\mathsf{sub}}$ from (\ref{eq.bijection, pseudo-envelops}) and $\Phi^{\mathsf{emb}}$ from (\ref{eq.bijection, envelops}) are generated as compositions of maps $p(\tilde v): (\d_1^+\hat X(\hat v),\, \d_2^-\hat X(\hat v)) \to (Y, \d Y)$ with the maps $\{\Phi^{\tilde\a^\d}: (Y, \d Y) \to \mathcal (\cP_d^{\mathbf c \Theta}, \cP_d^{(\emptyset)})\}$ from Theorem 3.2 in \cite{K9};  the latter map being  generated by the locus $\tilde\a(\d X) \subset \R \times Y$. Note that the cell $\cP_d^{(\emptyset)} \subset \cP_d$ contracts to a singleton $pt \in \cP_d^{(\emptyset)}$. 

Similarly, for any quasitopy $A: W \to \hat X \times [0,1]$ as in Definition \ref{quasi_isotopy_of_fields}, the product $Y \times [0,1]$ that encapsulates $\hat X \times [0,1]$ is mapped to $\cP_{d'}^{\mathbf{c}\Theta'}$ with the help of $A(\delta W)$, while 
$\d Y \times [0, 1]$ are mapped to the cell $\cP_{d'}^{(\emptyset)}$, which contracts to the singleton $\e_{d, d'}(pt) \in \cP_{d'}^{(\emptyset)}$.

As in Theorem 3.2 from \cite{K9}, and by similar transversality arguments, $\Phi^{\mathsf{emb}}$ is a 1-to-1 map. Here we need to use Lemma 3.4 from \cite{K9} 
to conclude that, under a $\d \mathcal E_d$-{\sf regular map} (see Definition 3.4 from \cite{K9}), the preimage of the hypersurface $\d \mathcal E_d \subset \R \times \cP_d$ (see (\ref{eq.E})) bounds a compact manifold $X$ in $\hat X$, provided $d \equiv 0 \mod 2$. Again, as in Proposition 3.5 from \cite{K9}, the bijective map in (\ref{eq.bijection, envelops}) helps to prove that the map in (\ref{eq.bijection, pseudo-envelops}) is a split surjective one.
\end{proof}

\begin{corollary} The sets $\mathcal{QT}^{\mathsf{emb}}_{d, d'}(\hat X, \hat v; \mathbf{c}\Theta; \mathbf{c}\Theta')$ are trivial when $\dim(\hat X) < \max_{\om \in \Theta} \{|\om|'\}$.
\end{corollary}

\begin{proof} Since $\textup{codim}(\cP_d^\Theta, \cP_d) = \max_{\om \in \Theta} \{|\om|'\}$ and $\cP_d$ is contractible,  by Theorem \ref{th.envelops}, the claim follows by a general position argument,  applied to homotopies of maps from the pair $(\d_1^+\hat X(\hat v),\, \d_2^-\hat X(\hat v))$ to the pair $(\cP_d^{\mathbf c\Theta}, pt)$. 
\end{proof}

\begin{corollary}\label{cor.codim}
The sets $\mathcal{QT}^{\mathsf{emb}}_{d, d'}(\hat X, \hat v; \mathbf{c}\Theta; \mathbf{c}\Theta')$, $\mathcal{QT}^{\mathsf{sub}}_{d, d'}(\hat X, \hat v; \mathbf{c}\Theta; \mathbf{c}\Theta')$ are invariants of the path-connected component of the vector field $\hat v$ in the space $\mathsf{conv}(\hat X)$ of convex traversing vector fields.
\end{corollary}

\begin{proof} 
By their definitions, the sets $\mathcal{QT}^{\mathsf{emb}}_{d, d'}(\hat X, \hat v; \mathbf{c}\Theta; \mathbf{c}\Theta')$ and $\mathcal{QT}^{\mathsf{sub}}_{d, d'}(\hat X, \hat v; \mathbf{c}\Theta; \mathbf{c}\Theta')$ depend only of the smooth topological types of the oriented $1$-dimensional foliations $\mathcal L(\hat v)$, and $\mathcal L^\bullet(\hat v^\bullet)$ on $\hat X$ and $\hat X \times [0,1]$, respectively. By Lemma \ref{lem.conv_fields}, the smooth topological types of these foliations do not change along any path in the space $\mathsf{conv}(\hat X)$, that contains the point $\hat v$. 
\end{proof}
\begin{corollary}\label{cor.depends_on_htype} The set $\mathcal{QT}^{\mathsf{emb}}_{d, d'}(\hat X, \hat v; \mathbf{c}\Theta; \mathbf{c}\Theta')$ depends only on the homotopy type of the pair $(\d_1^+\hat X(\hat v),\, \d_2^-\hat X(\hat v))$.
\end{corollary}

\begin{proof} The claim follows instantly from the bijection in formula (\ref{eq.bijection, envelops}). 
\end{proof}

The next theorem is our main result about the stability of $\mathcal{QT}^{\mathsf{emb}}_{d,\, d+2}(\hat X, \hat v; \mathbf{c}\Theta; \mathbf{c}\Theta)$ as a function of $d$ in terms of the function $\eta_\Theta(d+2)$ from (\ref{eq.eta(Theta)}).
\begin{theorem}\label{th.envelop_stabilization} {\bf (short stabilization: $\mathbf{\{d \Rightarrow d+2\}}$)} 
Let $\Theta$ be a closed subposet of $\mathbf \Om$. If 
$$\dim(\hat X) <  d + 2 - \psi_\Theta(d+2)$$  and $|\om|' > 2$ for all $\om \in \Theta_{\langle d+2]}$, then there exists a bijection 
\begin{eqnarray}\label{BLA}
\Phi^{\mathsf{emb}}: \mathcal{QT}^{\mathsf{emb}}_{d,\, d+2}(\hat X, \hat v; \mathbf{c}\Theta; \mathbf{c}\Theta) \stackrel{\approx}{\longrightarrow} 
 \big[(\d_1^+\hat X(\hat v),\, \d_2^-\hat X(\hat v)),\, (\cP_d^{\mathbf c\Theta}, pt)\big] \nonumber
\end{eqnarray}
\end{theorem}

\begin{proof} Let $k =_{\mathsf{def}} d+2 - \psi_\Theta(d+2)$. If $|\om|' > 2$ for all $\om \in \Theta_{\langle d+2]}$, then both $\cP_d^{\mathbf c\Theta}$ and $\cP_{d+2}^{\mathbf c\Theta}$ are simply-connected, since $\textup{codim}(\bar{\cP}_{d+2}^{\Theta}, \bar{\cP}_{d+2}) > 2$ and $\textup{codim}(\bar{\cP}_{d}^{\Theta}, \bar{\cP}_{d}) > 2$. By combining Theorem \ref{th.main_stab} with the Alexander duality as in \cite{KSW2}, we get an isomorphism $\e_\ast: H_j(\mathcal P_{d}^{\mathbf c\Theta}; \Z) \approx H_j(\mathcal P_{d+2}^{\mathbf c\Theta}; \Z)$ for all $j < k$. By the Whitehead Theorem (the inverse Hurewicz Theorem) (see Theorem (7.13) in \cite{Wh} or Theorem 10.1 in \cite{Hu}),
the map $\e: \mathcal P_{d}^{\mathbf c\Theta} \to \mathcal P_{d+2}^{\mathbf c\Theta}$ is $(k-1)$-connected. Thus, if $n = \dim(\d_1^+\hat X(\hat v)) \leq k-1$, then,  by the standard application of the obstruction theory, no map $\Phi: (\d_1^+\hat X(\hat v),\, \d_2^-\hat X(\hat v)) \to (\cP_d^{\mathbf c\Theta}, pt)$, which is not null-homotopic, becomes null-homotopic in $(\cP_{d+2}^{\mathbf c\Theta}, pt')$. 
Now the claim follows from Theorem \ref{th.envelops}.
\end{proof}

Theorem \ref{th.envelop_stabilization} leads to the following straightforward, but important implication.

\begin{corollary}\label{cor.STABILIZATION} {\bf (long stabilization $\mathbf{d \Rightarrow \infty}$)} Let $\Theta$ be a 
closed profinite (see Definition \ref{def.profinite}) poset such that 
$|\om|' > 2$ for all $\om \in \Theta$.  

Then, given a convex pair $(\hat X, \hat v)$, the quasiptopy set $\mathcal{QT}^{\mathsf{emb}}_{d,\, d'}(\hat X, \hat v; \mathbf{c}\Theta; \mathbf{c}\Theta)$ stabilizes towards  the set of homotopy classes $\big[(\d_1^+\hat X(\hat v),\, \d_2^-\hat X(\hat v)),\, (\cP_\infty^{\mathbf c\Theta}, pt)\big]$ for all sufficiently big $d$ relative to $\dim(\hat X)$. 
\hfill $\diamondsuit$
\end{corollary}

\subsection{Group structure on the ball-shaped convex pseudo-envelops}

The connected sum $\a_1 \uplus \a_2$ of convex pseudo-envelops $\a_1: X_1 \to \hat X_1, \a_2: X_2 \to \hat X_2$ can be introduced in a fashion, similar to the operation $\uplus$ in formulae (3.14) and (3.15) from \cite{K9} and formula (\ref{eq_cup_sub}) above. 

Let $D^{n-1}_+$ denote the Southern hemisphere in $\d D^n$. As with the connected sums of submersions from Section 3, there is an ambiguity about how to attach a $1$-handle $D^{n-1}_+ \times [0, 1]$ to $\d_1^+\hat X_1(\hat v_1) \coprod \d_1^+\hat X_2(\hat v_2)$ and a $1$-handle $H = D^n \times [0, 1]$ to $\hat X_1 \coprod \hat X_2$
 to form the ``coordinated" connected sums $\d_1^+\hat X_1(\hat v_1)\, \#_\d\, \d_1^+\hat X_2(\hat v_2)$ and $\hat X_1\, \#_\d\, \hat X_2$. The ambiguity arises if $\d_2^-\hat X_1(\hat v_1)$ or $\d_2^-\hat X_2(\hat v_2)$ has more than a single connected component. To avoid it, as in \cite{K9}, formula (3.14),  we need to pick a preferred connected component of $\d_2^-\hat X_1(\hat v_1)$ and $\d_2^-\hat X_2(\hat v_2)$.
Using the local models of convex vector fields (as in the proof of Lemma \ref{lem.box}), the vector fields $\hat v_1$ and $\hat v_2$ extend across $H$ so that the convexity of the extended traversing vector field $\hat v_{\#_\d}$ is enforced. For example, the lower diagram in Fig.2 is the connected sum of the upper diagram with itself.

\begin{theorem}\label{th.envelops on disks}
Let $\Theta' \subset \Theta \subset \mathbf\Om$ be closed subposets which do not contain the element $(\emptyset)$.  Assume that $d' \geq d$ and $d' \equiv d \equiv 0 \mod 2$. Let  $\hat v$ be a convex traversing vector field on the standard ball $D^{n+1}$ such that 
 $\d_1^+D^{n+1}(\hat v)$ is diffeomorphic to the standard ball $D^n$.\footnote{The constant vector field will do.}\smallskip

The group operation in the sources of the maps (\ref{bijection, pseudo-envelops on disks}) and (\ref{bijection, envelops on disks}) below is the connected sum $\uplus$ of convex envelops/convex pseudo-envelops.\smallskip

\noindent $\bullet$ There is a  group isomorphism
\begin{eqnarray}\label{bijection, pseudo-envelops on disks}
\Phi^{\mathsf{emb}}: \mathcal{QT}^{\mathsf{emb}}_{d,\, d'}(D^{n+1}, \hat v;\, \mathbf{c}\Theta; \mathbf{c}\Theta') \stackrel{\approx}{\longrightarrow} \nonumber \\
\stackrel{\approx}{\longrightarrow} \big[\big[(D^n,\, S^{n-1}),\; \e_{d,\, d'}: (\cP_d^{\mathbf c\Theta}, pt) \to (\cP_{d'}^{\mathbf c\Theta'}, pt')\big]\big].
\end{eqnarray}
The group homomorphism 
\begin{eqnarray}\label{bijection, envelops on disks}
\Phi^{\mathsf{sub}}: \mathcal{QT}^{\mathsf{sub}}_{d,\, d'}(D^{n+1}, \hat v;\, \mathbf{c}\Theta; \mathbf{c}\Theta') \stackrel{}{\longrightarrow} \nonumber \\
\stackrel{epi}{\longrightarrow} \big[\big[(D^n,\, S^{n-1}),\; \e_{d,\, d'}: (\cP_d^{\mathbf c\Theta}, pt) \to (\cP_{d'}^{\mathbf c\Theta'}, pt')\big]\big].
\end{eqnarray}
is an split epimorphism. Moreover, we have a split group extension
\begin{eqnarray}\label{eq.splitting_ker}
\qquad 1 \to \ker(\Phi^{\mathsf{sub}}) \to \mathcal{QT}^{\mathsf{sub}}_{d,\, d'}(D^{n+1}, \hat v;\, \mathbf{c}\Theta; \mathbf{c}\Theta') \stackrel{\mathcal R}{\rightarrow} \mathcal{QT}^{\mathsf{emb}}_{d,\, d'}(D^{n+1}, \hat v;\, \mathbf{c}\Theta; \mathbf{c}\Theta') \to 1.
\end{eqnarray} 

\noindent $\bullet$ For $n > 1$ all these groups are abelian. \smallskip

\noindent $\bullet$ For $n  < \max_{\om \in \Theta_{\langle d]}} \{|\om|'\}$, the groups  $\mathcal{QT}^{\mathsf{emb}}_{d, d'}(D^{n+1}, \hat v;\, \mathbf{c}\Theta; \mathbf{c}\Theta')$ are trivial. 
\end{theorem}

\begin{proof} The main observation is that the construction of the boxes $(\hat X, \hat v) \subset ([0, 1] \times Y, \tilde v)$ in Lemma \ref{lem.box} is amenable to the connected sum operation for convex pseudo-envelops of traversing flows. That is, given two boxes $ ([0, 1] \times Y_1, \tilde v_1) \supset (\hat X_1, \hat v_1)$  and $([0, 1] \times Y_2, \tilde v_2) \supset (\hat X_2, \hat v_2)$ as in Lemma \ref{lem.box}, we get that 
$$(\hat X_1 \, \#_\d \, \hat X_2, \hat v_1 \, \#_\d \, \hat v_2) \subset ([0, 1] \times (Y_1 \, \#_\d \, Y_2) , \tilde v_1 \, \#_\d \, \tilde v_2)$$
is also a box as in that lemma. 
Therefore, all the constructions and arguments from Section 3 in \cite{K9} (like Proposition 3.4, 
Corollary 3.3, 
and Theorem 3.2) 
apply to the convex pseudo-envelops/envelops $\a: (X, \a^\dagger(\hat v)) \to (D^{n+1}, \hat v)$, where the traversing vector field $\hat v$ is convex with respect to $\d D^{n+1}$. Thus, as in formula (\ref{eq_cup_sub}) above (see formula (3.15) and Proposition 3.2 from \cite{K9}), 
for a  convex vector field $\hat v$ with $\d_1^+D^{n+1}(\hat v)$ being a smooth $n$-ball, the quasitopies 
$\mathcal{QT}^{\mathsf{emb}}_{d, d'}(D^{n+1}, \hat v;\, \mathbf{c}\Theta; \mathbf{c}\Theta')$  and $\mathcal{QT}^{\mathsf{sub}}_{d, d'}(D^{n+1}, \hat v;\, \mathbf{c}\Theta; \mathbf{c}\Theta')$ are groups. For $n > 1$ they are commutative by arguments as in Proposition 3.2 from \cite{K9}.

The last claim follows by the general position argument. 
\end{proof}

Recall that, by Corollary \ref{cor.depends_on_htype}, the set $\mathcal{QT}^{\mathsf{emb}}_{d, d'}(\hat X, \hat v; \mathbf{c}\Theta; \mathbf{c}\Theta')$ depends only on the homotopy type of the pair $(\d_1^+\hat X(\hat v),\, \d_2^-\hat X(\hat v))$; however, the corollary does not make any claims about $\mathcal{QT}^{\mathsf{sub}}_{d, d'}(\hat X, \hat v; \mathbf{c}\Theta; \mathbf{c}\Theta')$. The next proposition, in line with Proposition \ref{prop.homology_sphere} (which deals with the homology $n$-spheres, $n \geq 6$), is a hint that $\mathcal{QT}^{\mathsf{sub}}_{d, d'}(\hat X, \hat v; \mathbf{c}\Theta; \mathbf{c}\Theta')$ may also depend only on the homotopy type of the pair $(\d_1^+\hat X(\hat v),\, \d_2^-\hat X(\hat v))$.

\begin{proposition}\label{prop.choice_of_convex_field} For $n \geq 5$, assuming that the locus $\d_2^-D^{n+1}(\hat v)$ is simply-connected, the groups $\mathcal{QT}^{\mathsf{sub/emb}}_{d, d'}(D^{n+1}, \hat v;\, \mathbf{c}\Theta; \mathbf{c}\Theta')$ do not depend on the choice of the convex traversing vector field $\hat v$ on $D^{n+1}$. 

For $n \leq 3$, the groups $\mathcal{QT}^{\mathsf{sub/emb}}_{d, d'}(D^{n+1}, \hat v;\, \mathbf{c}\Theta; \mathbf{c}\Theta')$ also do not depend on the convex $\hat v$.
\end{proposition}

\begin{proof}
For a convex $\hat v$, the locus $\d_1^+\hat X(\hat v)$ is a deformation retract of $D^{n+1}$ and thus has a homotopy type of a point. By the Poincar\'{e} duality,  $\d_2^-D^{n+1}(\hat v) = \d(\d_1^+\hat X(\hat v))$ is a homology $(n-1)$-sphere $\Sigma^{n-1} \subset \d D^{n+1}$. 
Let us delete a small smooth ball $B^n$ from the interior of $\d_1^+\hat X(\hat v)$. We denote its complement by $W^n$. For $n \geq 5$, assuming that $\pi_1(\d_2^-D^{n+1}(\hat v)) =1$, we may apply the smooth $h$-cobordism theorem (see \cite{Mi}) to $W$ to conclude that it is diffeomorphic to the product $S^{n-1} \times [0, 1]$. Therefore, $\d_1^+D^{n+1}(\hat v) $ is a smooth ball $D^n \subset S^n$. Any two regular embeddings $D^n \hookrightarrow S^n$ are diffeotopic. By the proof of Lemma \ref{lem.conv_fields}, any diffeotopy of $S^n$ which maps $\d_1^+D^{n+1}(\hat v_1) $ to $\d_1^+D^{n+1}(\hat v_2)$ extends to a diffeotopy of $D^{n+1}$ that maps the $\hat v_1$-trajectories to $\hat v_2$-trajectories, while preserving their orientation. As a result, for $n \geq 5$, the group $\mathcal{QT}^{\mathsf{emb}}_{d, d'}(D^{n+1}, \hat v;\, \mathbf{c}\Theta; \mathbf{c}\Theta')$ does not depend on the choice of $\hat v$, as long as $\pi_1(\d_2^-D^{n+1}(\hat v)) =1$.
\smallskip

In small dimensions $n \leq 3$, the contractibility of $\d_1^+D^{n+1}(\hat v)$ implies that it is the standard smooth $n$-ball. The case $n =1$ is obvious. For $n=2$, if the domain $\d_1^+\hat X(\hat v) \subset S^2$ has a boundary that is a homology $1$-sphere, then the domain is the $2$-ball, and any $2$-balls in $S^2$ are isotopic.
 For $n=3$, if the domain $\d_1^+D^{4}(\hat v) \subset S^3$ has a smooth boundary that is a homology $2$-sphere. By the classification of $2$-surfaces, it follows that $\d_2^-D^{4}(\hat v)$ is the standard $2$-sphere. By the solution of the $3$-dimensional Poincar\'{e} Conjecture \cite{P1}-\cite{P3}, the contractible domain $\d_1^+D^{4}(\hat v)$ is with the spherical boundary is the $3$-ball, and any two $3$-balls in $S^3$ are isotopic. 
Thus, by Lemma \ref{lem.conv_fields}, the smooth isotopy type of the locus $\d_1^+D^{n+1}(\hat v) \subset S^n$ determines the smooth topological type of the foliation $\mathcal L(\hat v)$. Since, any two standard $n$-balls $\d_1^+\hat X(\hat v_1)$ and $\d_1^+\hat X(\hat v_2)$ are isotopic in $S^n$, the  two statements of the proposition are validated. 
The difficult case $n=4$ is wide open.
\end{proof}

Combining Theorem \ref{th.envelops} and Theorem \ref{th.envelops on disks}, we get the following claim.

\begin{corollary}\label{WORK} 
Let $\hat v^\parallel$ be a constant vector field on the ball $D^{n+1} \subset \R^{n+1}$, and let $\hat v$ be a convex vector field on $\hat X$. 
Put $r = \#(\pi_0(\d_2^-\hat X(\hat v)))$. With the help of the maps from (\ref{bijection, pseudo-envelops on disks}) and (\ref{bijection, envelops on disks}), the groups $$\mathcal G^{\mathsf{emb}} := \big (\mathcal{QT}^{\mathsf{emb}}_{d, d'}(D^{n+1}, \hat v^\parallel;\, \mathbf{c}\Theta; \mathbf{c}\Theta')\big )^r\; \text{ and } \; \mathcal G^{\mathsf{sub}} := \big(\mathcal{QT}^{\mathsf{sub}}_{d, d'}(D^{n+1}, \hat v^\parallel;\, \mathbf{c}\Theta; \mathbf{c}\Theta')\big )^r$$ are represented in the group 
$$\big( \pi_n(\cP_d^{\mathbf c\Theta}, pt)\big/ \ker\big\{(\e_{d, d'})_\ast: \pi_n(\cP_d^{\mathbf c\Theta}, pt) \to \pi_n(\cP_{d'}^{\mathbf c\Theta'}, pt')\big\}\big)^r.$$ We denote by $\Psi^{\mathsf{emb}}$ and $\Psi^{\mathsf{sub}}$ these two representations (the first one is an isomorphism).\smallskip

Then the maps $\Phi^{\mathsf{sub}}(\hat X, \hat v)$ in (\ref{eq.bijection, pseudo-envelops}) and $\Phi^{\mathsf{emb}}(\hat X, \hat v)$ in (\ref{eq.bijection, envelops}) from Theorem \ref{th.envelops} are \emph{equivariant} with respect to the $\mathcal G^{\mathsf{sub}}$- and $\mathcal G^{\mathsf{emb}}$-actions on their source sets and the $\Psi^{\mathsf{sub}}(\mathcal G^{\mathsf{sub}})$- and $\Psi^{\mathsf{emb}}(\mathcal G^{\mathsf{emb}})$-actions on their target sets. \hfill $\diamondsuit$
 \end{corollary}

\subsection{Quasitopies of envelops with generic 
combinatorics $\Theta$ and $d' = d$}
From now and until Subsection 4.9, each result about quasitopies $\mathcal{QT}^{\mathsf{sub/emb}}_{d, d'}(\hat X, \hat v; \mathbf{c}\Theta; \mathbf{c}\Theta')$ of convex envelops is a recognizable immage of a similar result from \cite{K9} about the quasitopies  $\mathcal{QT}^{\mathsf{imm/emb}}_{d, d'}(Y, \d Y; \mathbf{c}\Theta, (\emptyset), \mathbf{c}\Theta')$ of immersions/embedding into the products $\R \times Y$.

\begin{theorem}\label{th.traversally_generic_fields}   Let $(\hat X, \hat v)$ be a convex pair. Let $d, k$ be natural numbers such that  $2 < k < d$ and $d \equiv 0 \mod 2$. 
Put $\mathbf \Om_{|\sim|' \leq k-1} =_{\mathsf{def}} \mathbf c \mathbf \Om_{|\sim|' \geq k}$.
Then we get a bijection\footnote{see (\ref{eq.A-bouquet}) for the definition of the number $A(d, k) := A(d, k, 0)$}
\begin{eqnarray}\label{envelops, bouquet} 
\Phi: \; \mathcal{QT}^{\mathsf{emb}}_{d, d}(\hat X, \hat v;\, \mathbf \Om_{|\sim|' \leq k-1};  \mathbf \Om_{|\sim|' \leq k-1})\; \approx \; \big[\big(\d_1^+\hat X(\hat v),\, \d_2^-\hat X(\hat v)\big),\; \big(\bigvee_{\ell =1}^{A(d,\, k)} S^{k-1}_\ell,\; \star \big)\big]. \nonumber
\end{eqnarray}
\end{theorem}

\begin{proof} The theorem is based on Proposition 3.10 from \cite{K9}.  For $d > k > 2$, Proposition \ref{prop.skeleton}, being combined with the Alexander duality, describes the homotopy type of $\mathcal P_d^{\mathbf \Om_{|\sim|' \leq k-1}} = \mathcal P_d^{\mathbf{c\Om}_{|\sim|' \geq k}}$ as a bouquet $\bigvee_{\ell =1}^{A(d,\, k)} S^{k-1}_\ell$ of $(k-1)$-spheres. The space $\mathcal P_d^{(\emptyset)}$ is contractible to the point $\star$. Thus, by Theorem \ref{th.envelops} (based on Theorem 3.1 from \cite{K9}),  
the claim follows.  
\end{proof}

We consider now the ``combinatorially generic" case of convex pseudo-envelops $\a: (X, \a^\dagger(\hat v)) \to (\hat X, \hat v)$ for which $\hat v$ is {\sf traversally generic} with respect to $\a(\d X)$.  By Definition \ref{def.traversally_generic}, any such $\a$ has tangency patterns that belong to the poset $\mathbf \Om_{|\sim|' \leq n}$, where $n = \dim \hat X -1$. 

\begin{corollary}\label{cor.submersions_bouquets} Let $(\hat X, \hat v)$ be a convex pair, $\dim \hat X = n+1$. For $d > n > 2$, we get a bijection  
$$\Phi: \, \mathcal{QT}^{\mathsf{emb}}_{d, d}(\hat X, \hat v;\, \mathbf \Om_{|\sim|' \leq n};  \mathbf \Om_{|\sim|' \leq n}) \approx \Z^{A(d,\, n+1)}.$$ 
These $\Z$-valued invariants of convex envelops are delivered by the degrees of the maps $$\{\Phi_\ell: \d_1^+\hat X(\hat v)\big/\d_2^-\hat X(\hat v) \to S^{n}_\ell\}_{\ell \in [1,\; A(d, n+1)]},$$ to the individual spheres, induced by the map $\Phi$ from Theorem \ref{th.traversally_generic_fields}. 

In particular, the $(d,d; \mathbf \Om_{|\sim|' \leq n}, \mathbf \Om_{|\sim|' \leq n})$-quasitopy class of any \emph{traversally generic} convex envelop $\a: (X, \a^\dagger(\hat v))  \to (\hat X, \hat v)$ is determined by the collection of such degrees.
\end{corollary}

\begin{proof} Since $\hat X$ is connected and $\hat v$ is convex, the locus $\d_1^+\hat X$ is connected as well. 

We repeat  the arguments from Proposition 3.6 in \cite{K9}. 
For $n > 2$, the homotopy classes of the classifying maps 
$\Phi^\a:  (\d_1^+\hat X(\hat v),\, \d_2^-\hat X(\hat v)) \to \big(\mathcal P_d^{\mathbf \Om_{|\sim|' \leq n}},\, \mathcal P_d^{(\emptyset)}\big)$ 
are in $1$-to-$1$ correspondence with homotopy classes of the corresponding maps 
$\tilde\Phi^\a:  (\d_1^+\hat X(\hat v),\, \d_2^-\hat X(\hat v)) \to \big(\bigvee_{\ell =1}^{A(d,\, n+1)} S^{n}_\ell,\, \star \big)$. For $n \geq 2$, the latter ones are detected by the degrees of the maps $\{\Phi^\a_\ell\}$ to the individual spheres $\{S^{n}_\ell\}$.
\end{proof}

\begin{example} \emph{Take $d= 6$ and $n=3$. Then $\bar{\mathcal P}_6^{\mathbf\Om_{|\sim|' \geq 4}}$ consist of a single $0$-dimensional cell (the $``\infty" \in \bar{\mathcal P}_6$), one $1$-dimensional cell, labelled by $\om = (6)$, and five $2$-dimensional cells, labelled by $\om = (51), (15), (42), (24), (33)$. Hence, $A(6, 4) = \chi(\bar{\mathcal P}_6^{\mathbf\Om_{|\sim|' \geq 4}}) -1 = 4$. The space $\bar{\mathcal P}_6^{\mathbf\Om_{|\sim|' \geq 4}}$ has a homotopy type of a bouquet $\bigvee_{\ell =1}^4 S^2$ of four $2$-spheres. By the Alexander duality, $\mathcal P_6^{\mathbf c\mathbf\Om_{|\sim|' \geq 4}} = \mathcal P_6^{\mathbf\Om_{|\sim|' \leq 3}}$ has a homotopy type of a bouquet $\bigvee_{\ell =1}^4 S^3$ of four $3$-spheres. }

\emph{We pick a convex and traversing vector field $\hat v$ on $D^4$. Recall that $\d_1^+D^4(\hat v)$ is contractible and thus $\d_2^-D^4(\hat v)$ a homology $2$-sphere, which implies that $\d_2^-D^4(\hat v)$ is diffeomorphic to $S^2 \subset S^3$.  Thus $\d_1^+D^4(\hat v)$, by \cite{P1}, \cite{P2}, is diffeomorphic to the ball $D^3$. Therefore, for any convex $\hat v$, we get the group isomorphism 
$$\mathcal{QT}^{\mathsf{emb}}_{6, 6}(D^4, \hat v;\, \mathbf \Om_{|\sim|' \leq 3};  \mathbf \Om_{|\sim|' \leq 3}) \approx \pi_3(\bigvee_{\ell =1}^4 S^3_\ell) \approx  \Z^4. \quad 
$$
As a result, any element $[\a] \in \mathcal{QT}^{\mathsf{sub}}_{6, 6}(D^4, \hat v;\, \mathbf \Om_{|\sim|' \leq 3};  \mathbf \Om_{|\sim|' \leq 3})$ generates four integer-valued characteristic invariants. For embeddings $\a$, they determine $[\a]$.}

\emph{At the same time,  by Corollary \ref{cor.n,n+1} below, $\mathcal{QT}^{\mathsf{emb}}_{6, 8}(D^4, \hat v;\, \mathbf \Om_{|\sim|' \leq 3};  \mathbf \Om_{|\sim|' \leq 4}) = 0$.}
\hfill $\diamondsuit$
\end{example}

The following claim contrasts  Corollary \ref{cor.submersions_bouquets}.

\begin{corollary}\label{cor.n,n+1} Under the hypotheses of Corollary \ref{cor.submersions_bouquets}, including $d > n > 2$, the set $\mathcal{QT}^{\mathsf{emb}}_{d,\, d+2}(\hat X, \hat v;\, \mathbf \Om_{|\sim|' \leq n}; \mathbf \Om_{|\sim|' \leq n+1})$ consists of a single element, represented by $X = \emptyset$. \smallskip

In particular, the $(d,d+2; \mathbf \Om_{|\sim|' \leq n}, \mathbf \Om_{|\sim|' \leq n+1})$-quasitopy class of any traversally generic convex envelop $\a: (X, \a^\dagger(\hat v))  \to (\hat X, \hat v)$ is trivial.
\end{corollary}

\begin{proof} Consider a convex pseudo-envelop $\a: (X, \a^\dagger(\hat v)) \to (\hat X, \hat v)$ as in Theorem \ref{th.traversally_generic_fields}. 

The claim is based on the observation that any map from the $n$-dimensional $CW$-complex $\d_1^+\hat X(\hat v)/ \d_2^-\hat X(\hat v)$ to $\mathcal P_d^{\mathbf {c \Om}_{|\sim|' \geq n}}$---homotopically a  bouquet of $n$-spheres--- is null-homptopic in $\mathcal P_{d+2}^{\mathbf {c \Om}_{|\sim|' \geq n+1}}$, since the latter space has the homotopy type of a  bouquet of $(n+1)$-spheres.  Again, by Theorem \ref{th.envelops}, the pseudo-envelop $\a$ is null-quasitopic. 
\end{proof}

\subsection{From inner framed cobordisms of $\d_1^+\hat X(\hat v)$ to quasitopies of $k$-flat convex envelops}
Theorem \ref{th.traversally_generic_fields} suggests a somewhat unexpected connection between quasitopies of  certain convex envelops $\a: X \to \hat X$ and inner framed cobordisms of the manifold $\d\hat X$. \smallskip

Let $Y$ be a compact smooth $n$-manifold. For $k-1 \leq n$, we associate with $Y$ the set of {\sf inner framed cobordisms} $\mathcal{FB}^{k-1}_\Xi(Y)$. These cobordisms are based on codimension $k-1$ smooth  
closed submanifolds  $Z$ of $Y$ of the form $Z = \coprod_{\s=1}^{A(d, k, q)} Z_\s$, where the number $A(d, k, q)$ is introduced in (\ref{eq.A-bouquet}). 
The normal $(k-1)$-bundle $\nu(Z, Y)$ is required to be framed. The disjoint ``components" $\{Z_\s\}$ of $Z$ are marked with different {\sf colors} $\s$ from a {\sf pallet} $\Xi$ of cardinality $A(d, k, q)$.

 We have seen in \cite{K9}, Proposition 3.11, 
 that the inner framed $\Xi$-colored codimension $k-1$ bordisms of the $n$-dimensional manifold $Y$ produce, via the Thom construction, quite special $k$-flat embeddings $\b: M \to \R \times Y$, where $\dim M =\dim Y$. The analogous mechanism, with the help of Lemma \ref{lem.box} and Theorem \ref{th.traversally_generic_fields}, generates special $k$-{\sf flat} envelops (see Definition \ref{def.k-flat_envelops}). 

\begin{proposition}\label{prop.inner_bordisms, envelops} Let $\hat X$ be a smooth compact connected $(n+1)$-dimensional manifold, equipped with a convex traversing vector field $\hat v$.  Let 
$k \in [3, n+1]$, $k < d$, and $d \equiv 0 \mod 2$. 
Then the Thom construction  delivers a bijection
$$Th: \mathcal{FB}_\Xi^{k-1}(\d_1^+\hat X(\hat v)) \approx \mathcal{QT}^{\mathsf{emb}}_{d, d}(\hat X, \hat v;\, \mathbf c\mathbf \Om_{|\sim|' \geq k};  \mathbf c\mathbf \Om_{|\sim|' \geq k}),$$ where  $\mathcal{FB}_\Xi^{k-1}(\d_1^+\hat X(\hat v))$ denotes the set of inner framed $\Xi$-colored $(n - k+1)$-dimensional bordisms of the space $\d_1^+\hat X(\hat v)$. Here $\#\Xi = A(d, k)$ (see \ref{eq.A-bouquet}). \hfill $\diamondsuit$
\end{proposition}

\begin{example}\label{ex.Sigma10A} \emph{Let us recycle Example 3.5 from \cite{K9}: 
take $k=3$,  $d = 6$, $q = 0$, and $n = 3$. Then $\#\Xi= 10$. So we get a homotopy equivalence $\tau: \cP_6^{\mathbf{c}\Theta_{|\sim |' \geq 3}} \sim \bigvee_{\s=1}^{10} S^2_\s$. Let $(\hat X, \hat v)$ be a convex pair, where $\dim \hat X = 4$, and let $Z \subset \d_1^+\hat X(\hat v)$ be a normally framed $1$-dimensional closed submanifold, each loop in $Z$ being colored with a color from the pallet $\Xi$. Using the isomorphism 
$$Th: \mathcal{FB}^2_\Xi(\d_1^+\hat X(\hat v)) \approx \mathsf{QT}^{\mathsf{emb}}_{6, 6}\big(\hat X, \hat v; \mathbf{c}\Theta_{|\sim |' \geq 3}; \mathbf{c}\Theta_{|\sim |' \geq 3}\big),$$  any \emph{framed link} $Z \subset \d_1^+\hat X(\hat v)$, colored with $10$ distinct colors at most, produces a quasitopy class of a convex envelop   $\a: (X, \a^\dagger(\hat v)) \to  (\hat X, \hat v)$. Its tangency patterns $\om$, except for $\om = (\emptyset)$, have only entries from the list $\{1, 2, 3\}$, so that no more than one $3$ is present in $\om$, and  no more than two $2$'s are present, while $|\om| \leq 6$.} 
\smallskip

\emph{Now, let us assume that $D^4$ carries a constant vector field $v^\parallel$ and that $\d_2^-\hat X(\hat v) \neq \emptyset$. Note that any element of $H_1(\d_1^+\hat X(\hat v)); \Z)$ my be realized by a disjoint union of framed oriented loops. Then, like in Example  3.5 from \cite{K9}, 
the orbit-space of the $\mathsf{G}^{\mathsf{emb}}_{6,6}(D^4; \mathbf{c}\Theta_{|\sim |' \geq 3}; \mathbf{c}\Theta_{|\sim |' \geq 3})$-action on $\mathsf{QT}^{\mathsf{emb}}_{6, 6}\big(\hat X, \hat v; \mathbf{c}\Theta_{|\sim |' \geq 3}; \mathbf{c}\Theta_{|\sim |' \geq 3}\big)$ admits a surjection on the group $(H_1(\d_1^+\hat X(\hat v)); \Z))^{10} \approx (H_1(\hat X; \Z))^{10}$, provided that $\d\hat X$ is orientable.} \smallskip

\emph{Let us give a couple of specific examples of this fact. Consider the box $T^3_\circ \times [0, 1]$, where $T^3_\circ$ is the compliment in the $3$-torus $T^3$ to a ball $D^3$. By rounding the corners of the box, we get a convex pair $(\hat X, \hat v)$, where $\hat X$ is homeomorpfic to $T^3_\circ \times [0, 1]$. Then the orbit-space of the $\mathsf{G}^{\mathsf{emb}}_{6,6}(D^4; \mathbf{c}\Theta_{|\sim |' \geq 3}; \mathbf{c}\Theta_{|\sim |' \geq 3})$-action on  $\mathsf{QT}^{\mathsf{emb}}_{6, 6}\big(\hat X, \hat v;  \mathbf{c}\Theta_{|\sim |' \geq 3}; \mathbf{c}\Theta_{|\sim |' \geq 3}\big)$ is mapped onto the lattice $\Z^{30}$.} 

\emph{Let $M = \mathbb H^3 / \Gamma$ be a compact hyperbolic $3$-manifold and $M_\circ := M \setminus D^3$. By rounding the corners of the box $M_\circ \times [0, 1]$, we get a convex pair $(\hat X, \hat v)$. Then the orbit-space of the $\mathsf{G}^{\mathsf{emb}}_{6,6}(D^4; \mathbf{c}\Theta_{|\sim |' \geq 3}; \mathbf{c}\Theta_{|\sim |' \geq 3})$-action on  $\mathsf{QT}^{\mathsf{emb}}_{6, 6}\big(\hat X, \hat v;  \mathbf{c}\Theta_{|\sim |' \geq 3}; \mathbf{c}\Theta_{|\sim |' \geq 3}\big)$ is mapped onto the abelian group $(\Gamma/[\Gamma, \Gamma])^{10}$, where $[\Gamma, \Gamma]$ denotes the commutator.
\hfill $\diamondsuit$}
\end{example}

By Proposition \ref{prop.inner_bordisms, envelops} and repeating the (obviously modified) arguments in the proof of Corollary 3.13 from \cite{K9}, we get the following corollary. 

\begin{corollary}\label{cor.actions_II}
Let 
$d > k$, and $d \equiv 0 \mod 2$. Let $(\hat X, \hat v)$ be a convex pair.

\begin{itemize}
\item For $\dim \hat X = k \geq 3$ and any choice of $\kappa \in \pi_0(\d_2^-\hat X(\hat v))$,  the group \hfill \break $\mathcal{QT}^{\mathsf{emb}}_{d, d}(D^k, \hat v^\parallel;\, \mathbf c\mathbf \Om_{|\sim|' \geq k};  \mathbf c\mathbf \Om_{|\sim|' \geq k})$ acts freely and transitively on the set \hfill\break   $\mathcal{QT}^{\mathsf{emb}}_{d, d}(\hat X, \hat v;\, \mathbf c\mathbf \Om_{|\sim|' \geq k};  \mathbf c\mathbf \Om_{|\sim|' \geq k})$. Thus both sets are in a $1$-to-$1$ correspondence.  \smallskip
 
\item For $\dim \hat X = k+1 > 5$,  a simply-connected $\hat X$, and any choice of $\kappa \in \pi_0(\d_2^-\hat X(\hat v))$, the group $\mathcal{QT}^{\mathsf{emb}}_{d, d}(D^{k+1}, \hat v^\parallel;\, \mathbf c\mathbf \Om_{|\sim|' \geq k};  \mathbf c\mathbf \Om_{|\sim|' \geq k})$ acts freely and transitively on the set  $\mathcal{QT}^{\mathsf{emb}}_{d, d}(\hat X, \hat v;\, \mathbf c\mathbf \Om_{|\sim|' \geq k};  \mathbf c\mathbf \Om_{|\sim|' \geq k})$. Again, both sets are in a $1$-to-$1$ correspondence.  

\hfill $\diamondsuit$
\end{itemize}
\end{corollary}

\begin{example}\emph{ Let $\hat X = \C\P^2 \times D^2$ with $\hat v$ on $\C\P^2 \times D^2$ being generated by a constant vector field on $D^2$. Then $\mathcal{QT}^{\mathsf{emb}}_{d, d}(D^{6}, \hat v^\parallel;\, \mathbf c\mathbf \Om_{|\sim|' \geq 5};  \mathbf c\mathbf \Om_{|\sim|' \geq 5})$ acts freely and transitively on the set  $\mathcal{QT}^{\mathsf{emb}}_{d, d}(\C\P^2 \times D^2, \hat v;\, \mathbf c\mathbf \Om_{|\sim|' \geq 5};  \mathbf c\mathbf \Om_{|\sim|' \geq 5})$ for all even $d \geq 6$. Thus both sets are in $1$-to-$1$ correspondence. As a result, $\mathcal{QT}^{\mathsf{emb}}_{d, d}(\C\P^2 \times D^2, \hat v;\, \mathbf c\mathbf \Om_{|\sim|' \geq 5}; \mathbf c\mathbf \Om_{|\sim|' \geq 5})$ acquires  the  structure of the abelian group $\pi_5\big(\bigvee_{i=1}^{A(d, 5)}S^4_i,\, \star\big)$.}
\hfill $\diamondsuit$
\end{example}

\subsection{Convex envelops with special combinatorics $\Theta$}
For a given $\om \in \mathbf \Om$, we denote by $\langle \om \rangle$ the minimal closed poset, generated by $\om$. Combining Theorem \ref{th.envelops} and Theorem 3.5 from \cite{K9} with Proposition \ref{prop.choice_of_convex_field},  
we get the following result, in which the constraint $d \leq 12$ reflects only the scope of the numerical experiments in \cite{KSW2}.

\begin{theorem}\label{th.envelops,d<14} Assume that $d \leq 12$  and $d \equiv 0 \mod 2$.  Let $\om \in \mathbf\Om_{\leq d}$ be such that $|\om|' > 2$. 
 Let $\hat X$ be a smooth compact connected $(n+1)$-dimensional manifold, equipped with a convex traversing vector field $\hat v$.

Then the quasitopy set $\mathcal{QT}^{\mathsf{emb}}_{d, d}(\hat X, \hat v; \mathbf{c}\langle \om \rangle; \mathbf{c}\langle \om \rangle)$ either consists of single element (is trivial), or is isomorphic to the cohomotopy set $\pi^k(\d_1^+\hat X(\hat v)/\d_2^-\hat X(\hat v))$, where $k = k(\om) \in [|\om|' -1,\, d-1]$. 

In particular, for a constant vector field $\hat v = v^\parallel$, $d \leq 12$, and $\om$ and $k = k(\om)$ from the table in \cite{K9}, Appendix, the group $\mathcal{QT}^{\mathsf{emb}}_{d, d}(D^{n+1}, \hat v; \mathbf{c}\langle \om \rangle, \mathbf{c}\langle \om \rangle) \approx \pi_n(S^k)$. A similar claim is valid for any convex $\hat v$ on $D^{n+1}$, provided that either $n \leq 3$ or $n \geq 5$ and $\d_2^-\hat X(\hat v)$ is simply-connected.
\hfill $\diamondsuit$
\end{theorem}

For $d \leq 13$, the table from Appendix in \cite{K9} lists all $\om$'s and the corresponding $k = k(\om)$'s for which $\mathcal P_d^{\mathbf c \langle \om \rangle}$ is homologically nontrivial.  
In fact, $\mathcal P_d^{\mathbf c \langle \om \rangle}$ is a homology $k$-sphere and even a homotopy $k$-sphere, at least when  $|\om|' > 2$ \cite{K9}, a quite mysterious phenomenon... \smallskip

Let us recycle Example 3.4 from \cite{K9}, while adapting it to convex envelops.

\begin{example}\emph{Let $d =8$ and $\om = (4)$. Then, for any convex pair  $(\hat X, \hat v)$, using the list  in \cite{K9}, Appendix, we get a bijection $$\mathcal{QT}^{\mathsf{emb}}_{8,8}(\hat X, \hat v; \mathbf{c}\langle(4) \rangle, \mathbf{c}\langle(4) \rangle) \approx \pi^4\big(\d_1^+\hat X(\hat v)\big/ \d_2^-\hat X(\hat v)\big),$$ where $\pi^4(\sim)$ stands for the $4$-cohomotopy set. In particular, using Proposition \ref{prop.choice_of_convex_field}, for any convex $\hat v$ on $D^8$ such that the locus $\d_2^-D^8(\hat v)$ is simply-connected (the constant $\hat v$ will do), we get a group isomorphism: $$\mathcal{QT}^{\mathsf{emb}}_{8,8}(D^8, \hat v; \mathbf{c}\langle(4) \rangle, \mathbf{c}\langle(4) \rangle) \approx \pi_7(S^4) \approx \Z \times \Z_{12}.$$} 
\emph{Let $d =12$ and $\om = (11213)$. Then, for any for any convex pair  $(\hat X, \hat v)$, by the same list from \cite{K9}, we get a bijection $$\mathcal{QT}^{\mathsf{emb}}_{12,12}(\hat X, \hat v; \mathbf{c}\langle(11213) \rangle, \mathbf{c}\langle(11213) \rangle) \approx \pi^6\big(\d_1^+\hat X(\hat v)\big/ \d_2^-\hat X(\hat v)\big).$$ Again, using Proposition \ref{prop.choice_of_convex_field}, for any convex $\hat v$ on $D^{10}$, such that the locus $\d_2^-D^{10}(\hat v)$ is simply-connected (again, the constant $\hat v$ will do), we get a group isomorphism  $$\mathsf{QT}^{\mathsf{emb}}_{12,12}(D^{10}, \hat v; \mathbf{c}\langle(11213) \rangle, \mathbf{c}\langle(11213) \rangle) \approx \pi_9(S^6) \approx \Z_{24}. \quad \diamondsuit$$} 
\end{example}


Proposition \ref{prop.free_group_A} below deals with special $\Theta$'s for which $\cP_d^{\mathbf c\Theta}$ are $K(\pi, 1)$-spaces and $\pi$ is a free group. 
\smallskip

Let $H, G$ be two groups, and $\mathsf{Hom}(H, G)$ be the group of their homomorphisms. Then $G$ acts on $\mathsf{Hom}(H, G)$ by the conjugation: for any homomorphism $\phi: H \to G$,  $h \in H$, and $g \in G$, we define $(Ad_g \phi)(h)$ by the formula $g^{-1}\phi(h)g$. We denote by $\mathsf{Hom}^\bullet(H, G)$ the quotient $\mathsf{Hom}(H, G)/ Ad_G$.
\smallskip

\begin{proposition}\label{prop.free_group_A} Let $\mathbf c\Theta$ consist of all $\om$'s with entries $1$ and $2$ only and no more than a single entry $2$. Put $\kappa(d) =_{\mathsf{def}} \frac{d(d-2)}{4}$ for $d \equiv 0 \mod 2$. We denote by $\mathsf F_{\kappa(d)}$ the \emph{free} group in $\kappa(d)$ generators. \smallskip

Consider a convex pair $(\hat X, \hat v)$, where $\hat X$ is connected.
If $\d_2^-\hat X(\hat v) \neq \emptyset$, then there is a bijection 
$$\Phi^{\mathsf{emb}}: \mathcal{QT}^{\mathsf{emb}}_{d,d}(\hat X, \hat v; \mathbf{c}\Theta; \mathbf{c}\Theta) \stackrel{\approx}{\longrightarrow} \mathsf{Hom}(\pi_1(\hat X), \mathsf F_{\kappa(d)})
$$
and a surjection 
$$\Phi^{\mathsf{sub}}: \mathcal{QT}^{\mathsf{sub}}_{d,d}(\hat X, \hat v; \mathbf{c}\Theta; \mathbf{c}\Theta) \longrightarrow \mathsf{Hom}(\pi_1(\hat X), \mathsf F_{\kappa(d)}).$$

When $\d_2^-\hat X(\hat v) = \emptyset$, then similar claims hold with the target of $\Phi^{\mathsf{emb}}$ and $\Phi^{\mathsf{sub}}$ being replaced by the set $\mathsf{Hom}^\bullet(\pi_1(\hat X), \mathsf F_{\kappa(d)})$.

In particular, $\Phi^{\mathsf{emb}}: \mathcal{QT}^{\mathsf{emb}}_{d,d}(S^1 \times [0, 1], \hat v; \mathbf{c}\Theta, \mathbf{c}\Theta) \stackrel{\approx}{\longrightarrow} \mathsf F_{\kappa(d)}/ Ad_{\mathsf F_{\kappa(d)}}$, the free group of \emph{cyclic} words in $\kappa(d)$ letters
(see Fig.3 from \cite{K9}). Here $\hat v$ is tangent to the fibers of the obvious projection $S^1 \times [0, 1] \to S^1.$
\smallskip

If $\pi_1(\hat X)$ has no free images, then the group $\mathcal{QT}^{\mathsf{emb}}_{d,d}(\hat X, \hat v; \mathbf{c}\Theta, \mathbf{c}\Theta)$ is trivial. 
\end{proposition}

\begin{proof} We observe that $\pi_1(\d_1^+\hat X(\hat v) \approx \pi_1(\hat X)$ since the two spaces are homotopy equivalent due to the convexity of  $\hat v$. By Lemma \ref{lem.box}, we can incapsulate $(\hat X, \hat v)$ in a box $\R \times Y$, where $Y$ is homeomorphic to $\d_1^+\hat X(\hat v)$. Using this $Y$, the claim follows from Corollary 3.12 in \cite{K9}.
\end{proof}


The next proposition is a stabilization result, by the increasing $d' \geq d$, for the convex envelops with (using the terminology of \cite{Ar}),  $k$-{\sf moderate tangent patterns} $\om \in \mathbf{c}\Theta_{\max \geq k}$. The entries of such $\om$'s are all less than $k$. Proposition \ref{prop.moderate_QI} below follows instantly from  \cite{K9}, Proposition 3.8,  by combining it with Lemma \ref{lem.box}.

\begin{proposition}\label{prop.moderate_QI} Let $k \geq 4$. If $\dim Y \leq (k-2)(\lceil d/k \rceil +1)- 2$,
then the classifying map 
\begin{eqnarray}\label{emb_homotopy_AA}
\qquad\qquad \Phi^{\mathsf{emb}}: \mathcal{QT}^{\mathsf{emb}}_{d, d'}(\hat X, \hat v; \mathbf{c}\Theta_{\max \geq k}; \mathbf{c}\Theta_{\max \geq k}) \stackrel{\approx}{\longrightarrow} \nonumber 
 \big[(\d_1^+\hat X(\hat v),  \d_2^-\hat X(\hat v)), (\cP_d^{\mathbf c\Theta_{\max \geq k}}, pt) \big]
\end{eqnarray}
is a bijection, and the classifying map
\begin{eqnarray}\label{emb_homotopy_AAA}
\qquad\qquad \Phi^{\mathsf{sub}}: \mathcal{QT}^{\mathsf{imm}}_{d, d'}(\hat X, \hat v; \mathbf{c}\Theta_{\max \geq k}; \mathbf{c}\Theta_{\max \geq k}) \stackrel{\approx}{\longrightarrow} \nonumber 
 \big[(\d_1^+\hat X(\hat v),  \d_2^-\hat X(\hat v)), (\cP_d^{\mathbf c\Theta_{\max \geq k}}, pt) \big]
\end{eqnarray}
is a surjection for any $d' \geq d$, $d' \equiv d \mod 2$. \smallskip

In particular, for a given $(\hat X, \hat v)$, the quasitopy $\mathsf{QT}^{\mathsf{emb}}_{d, d'}(\hat X, \hat v; \mathbf{c}\Theta_{\max \geq k}; \mathbf{c}\Theta_{\max \geq k})$ stabilizes for all  $d' \geq d \geq \frac{k}{k-2}(\dim \hat X +3 -k)$, a linear function in $\dim \hat X$. \hfill $\diamondsuit$
\end{proposition}

\begin{example} \emph{Let $k=4$. By Proposition \ref{prop.moderate_QI}, for any compact connected $3$-\emph{dimensional} convex pair $(\hat X, \hat v)$, we get bijections:}
\begin{eqnarray}
\mathcal{QT}^{\mathsf{emb}}_{4, 4}(\hat X, \hat v; \mathbf{c}\Theta_{\max \geq 4}; \mathbf{c}\Theta_{\max \geq 4}) \stackrel{\approx}{\longrightarrow} \pi^2(\d_1^+\hat X(\hat v)\big/  \d_2^-\hat X(\hat v)),
\nonumber \\
\mathcal{QT}^{\mathsf{emb}}_{4, 6}(\hat X, \hat v; \mathbf{c}\Theta_{\max \geq 4}; \mathbf{c}\Theta_{\max \geq 4}) \stackrel{\approx}{\longrightarrow} \pi^2(\d_1^+\hat X(\hat v)\big/\d_2^-\hat X(\hat v)),\nonumber \\
\mathcal{QT}^{\mathsf{emb}}_{6, 6}(\hat X, \hat v; \mathbf{c}\Theta_{\max \geq 4}; \mathbf{c}\Theta_{\max \geq 4}) \stackrel{\approx}{\longrightarrow} \pi^2(\d_1^+\hat X(\hat v)\big/ \d_2^-\hat X(\hat v)), \nonumber
\end{eqnarray}
\emph{whose target is the second cohomotopy group $\pi^2(\d_1^+\hat X(\hat v)\big/ \d_2^-\hat X(\hat v))$ of the singular connected surface $\d_1^+\hat X(\hat v)\big/ \d_2^-\hat X(\hat v)$. This cohomotopy group is isomorphic to $\Z$ via the degree invariant. Thus, all the three types of quasitopy classes  are determined by this degree. 
Assuming $\d_2^-\hat X(\hat v) = \emptyset$, we may replace  $\pi^2(\d_1^+\hat X(\hat v)\big/ \d_2^-\hat X(\hat v))$ by $\pi^2(\hat X)$.}
\hfill $\diamondsuit$
\end{example}

\subsection{Characteristic classes of convex pseudo-evelops}
In this subsection, we will use the cohomology $H^\ast(\cP_d^{\mathbf c \Theta})$ of the classifying space $\cP_d^{\mathbf c \Theta}$ (see \cite{KSW2} and Section 2) to produce a variety of characteristic classes of convex pseudo-envelops.
\smallskip

Since any convex pseudo-envelop $\a: (X, \a^\dagger(\hat v)) \to (\hat X, \hat v)$ produces  an immersion $\a^\d: \d X \to \hat X$, the following theorem follows directly from  Theorem 3.3, \cite{K9}.

\begin{theorem}\label{th.A_car_classes} Let $(\hat X, \hat v)$ be a convex pair. Pick  $d' = d$, $d \equiv 0 \mod 2$, and $\Theta' = \Theta$.  

Then any convex pseudo-envelop $\a: (X, \a^\dagger(\hat v)) \to (\hat X, \hat v)$  induces a characteristic homomorphism $(\Phi^\a)^\ast$ from the $\ast$-homology of the differential complex $$\big\{(\d^\#)^\ast : \Z[\Theta_{\langle d]}^\#]^\ast \to \Z[\Theta_{\langle d]}^\#]^\ast\big\},$$ dual to the differential complex in (\ref{eq.quotient_complex}), to the cohomology $H^\ast(\hat X; \Z)$, and, via $\a^\ast$, further to the cohomology $H^\ast(X; \Z)$. 

The $(d, d; \mathbf c\Theta, \mathbf c\Theta)$-quasitopic pseudo-envelops/envelops induce the same characteristic homomorphisms. \hfill $\diamondsuit$
\end{theorem}

Let us revisit the Arnold-Vassiliev case \cite{Ar}, \cite{V} of real polynomials with {\sf moderate singularities}. Let $\Theta_{\max\, \geq k} \subset \mathbf\Om_{\langle d]}$ be the closed poset, consisting of $\om$'s with the maximal entry $\geq k$. For $k \geq 3$, the cohomology $H^j(\cP^{\mathbf c\Theta_{\max \geq k}}, pt; \Z)$ is isomorphic to $\Z$  in each dimension $j$ of the form $(k- 2)m$, where the integer $m \leq d/k$, and vanishes otherwise \cite{Ar}.
The cohomology ring $H^\ast(\cP^{\mathbf c\Theta_{\max \geq k}}, pt; \Z)$ was computed by Vassiliev in \cite{V}, Theorem 1 on page 87. Here is the summery of his result: consider the graded ring $\mathcal Vass_{d, k}$, multiplicatively generated over $\Z$ by the elements $\{e_m\}_{m \leq d/k}$ of the degrees $\deg(e_m) = m(k-2)$, subject to the relations
\begin{eqnarray}\label{eq.relations_ev} 
e_l \cdot e_m = \frac{(l+m)!}{l! \cdot m!}\; e_{l+m} \text{\; for } k\equiv 0 \mod 2, \text{ and the relations } 
\end{eqnarray}
\begin{eqnarray}
e_1\cdot e_1 = 0, \quad e_1\cdot e_{2m} = e_{2m+1}, \nonumber \
\end{eqnarray}
\begin{eqnarray}\label{eq.relations_odd}
e_{2l} \cdot e_{2m} = \frac{(l+m)!}{l! \cdot m!}\; e_{2l+2m} \text{\; for } k\equiv 1 \mod 2.
\end{eqnarray}

Combining Lemma \ref{lem.box} with Proposition 3.7 from \cite{K9}, we get the following assertion. 

\begin{proposition}\label{prop.Vass} Let $k \geq 3$. Consider a convex pseudo-envelop   $\a: (X, v) \to (\hat X, \hat v)$ whose tangency patterns to the  $\hat v$-flow belong to $\mathbf c\Theta_{\max \,\geq k} \subset \mathbf\Om_{\langle d]}$  (the tangencies are $k$-moderate). 

Then $\a$  generates a characteristic \emph{ring} homomorphism $$(\Phi^\a)^\ast: \mathcal Vass_{d, k} \to H^\ast(\d_1^+\hat X(\hat v), \d_2^-\hat X(\hat v); \Z),$$ 
which is an invariant of the quasitopy class of $\a$. 
In other words, we get a map $$\Phi^\ast_{d, k}: \mathcal{QT}^{\mathsf{imm}}_{d, d}(\hat X, \hat v; \mathbf{c}\Theta_{\max \geq k};  \mathbf{c}\Theta_{\max \geq k}) \to \mathsf{Hom_{ring}}\big(\mathcal Vass_{d, k},\; H^\ast(\d_1^+\hat X(\hat v), \d_2^-\hat X(\hat v); \Z)\big).$$

In particular, for any generator $e_l \in \mathcal Vass_{d, k}$,  
we get a characteristic element $(\Phi^\a)^\ast(e_l) \in H^{l(k-2)}(\d_1^+\hat X(\hat v), \d_2^-\hat X(\hat v); \Z)$ which is an invariant of the quasitopy class of $\a$. \smallskip

If $\d_2^-\hat X(\hat v) = \emptyset$, then $(\Phi^\a)^\ast(e_l)$ lives in 
$H^{l(k-2)}(\hat X; \Z) \approx H^{l(k-2)}(\d_1^+\hat X(\hat v); \Z)$. \smallskip

If $\hat X$ is oriented and $(n+1)$-dimensional, using the Poincar\'{e} duality in $\d_1^+\hat X(\hat v)$, we produce a homology class $\mathcal D((\Phi^\a)^\ast(e_l)) \in H_{n-l(k-2)}(\d_1^+\hat X(\hat v); \Z) \approx H_{n-l(k-2)}(\hat X; \Z)$ which is again an invariant of the quasitopy class of $\a$.
\hfill $\diamondsuit$
\end{proposition}

\begin{remark}
\emph{
We notice that if $(\Phi^\a)^\ast(e_l) = 0$, then, by (\ref{eq.relations_ev}) and (\ref{eq.relations_odd}), all $\{(\Phi^\a)^\ast(e_q)\}_{q \geq l}$ must be torsion elements in $H^{q(k-2)}(\d_1^+\hat X(\hat v), \d_2^-\hat X(\hat v); \Z)$. For an orientable $\hat X$, they may be viewed as elements of $H_{n-q(k-2)}(\hat X; \Z)$.} \hfill $\diamondsuit$
\end{remark}



%



\subsection{How to manufacture convex envelops with desired combinatorial tangency patterns}

Let us recall one classical construction, leading to the Alexander duality. Let $K \subset S^d$ be a $CW$-subcomplex of the $d$-sphere. For an element $a \in H_p(K; \Z)$, we denote by $c_a \in H_{p+1}(S^d, K; \Z)$ the unique element such that $\d_\ast(c_a) = a$. The Alexander duality operator $\mathcal A\ell$ is defined by the formula $\mathcal A\ell(a) =_{\mathsf{def}}  \mathcal D(c_a) \in H^{d-p-1}(S^d \setminus K; \Z)$, where $\mathcal D$ is the Poincar\'{e} duality operator, the inverse of the operator $[S^d]\cap$. Pick $b \in H_{d-p-1}(S^d \setminus K; \Z)$. Then {\sf linking number} of $a, b$ is defined by the formula $\mathsf{lk}(a, b) := \langle \mathcal A\ell(a),\, b \rangle$, where $\langle \sim, \sim \rangle$ is the natural pairing between cohomology and homology of dimension $n-p-1$.   \smallskip

The next lemma is instrumental in producing  examples of convex envelopes with prescribed combinatorial patterns of their trajectories.
\begin{lemma}\label{lem.char_classes} For any element $\om \in \mathbf\Om_{\langle d]}$, with the help of the differential $\d(\om)$, given by formula $(\ref{eq.d+d})$, the Alexander duality $\mathcal A\ell$ produces a cohomology class
$$\theta_\om =_{\mathsf{def}}\mathcal A\ell([\d\bar{\mathsf R}^\om]) \in\; H^{|\om|'}(\mathcal P_d^{\mathbf{c}\{\om_\succ\}}; \Z),$$
where $\om_\succ$ denotes the set of elements of $\mathbf\Om_{\langle d]}$ that are smaller than $\om$ (so, $\om \in \mathbf{c}\{\om_\succ\}$), and $\d\bar{\mathsf R}^\om$ denotes the algebraic boundary of the cell $\bar{\mathsf R}^\om$.
\end{lemma}

\begin{proof} For any $\om \in \mathbf\Om_{\langle d]}$, take the closed poset $\om_\succ =_{\mathsf{def}} \{\om' \prec \om\} \subset \mathbf\Om_{\langle d]}$ for the role of $\Theta$ in Theorem \ref{thA}. We denote by $\bar{\mathsf R}_d^\om$ the one-point compactification of the closed cell $\mathsf R_d^\om \subset \cP_d$ (the interior $(\mathsf R_d^\om)^\circ$ of $\mathsf R_d^\om$ is an open $(d -|\om|')$-ball). Then the differential  $\d(\omega)$, given by the formula (\ref{eq.d+d}), represents the $(d - |\om|' -1)$-cycle $\d\bar{\mathsf R}_d^\om$ in the chain complex $\mathcal C_\ast(\bar{\mathcal P}_d^{\,\om_\succ}; \Z)$ (note that $\d\bar{\mathsf R}_d^\om$ is a boundary in $\bar{\mathcal P}_d^{\,\om_\succeq}$, but not in $\bar{\mathcal P}_d^{\,\om_\succ}$ !) and thus defines a nontrivial element $[\d\bar{\mathsf R}_d^{\,\om}] \in H_{d - |\om|' -1}(\bar{\mathcal P}_d^{\,\om_\succ}; \Z)$. By the Alexander duality, this element produces a cohomology class $\theta_\om =_{\mathsf{def}} \mathcal A\ell([\d\bar{\mathsf R}_d^\om]) \in H^{|\om|'}(\mathcal P_d^{\mathbf{c}\{\om_\succ\}}; \Z).$
\end{proof}

\begin{example}\label{ex13.2} \emph{If  $d = 6$, we get the following cohomology classes:
\begin{itemize}
\item  $\theta_{(121)} = \mathcal A\ell\big(\bar{\mathsf R}_6^{(31)} - \bar{\mathsf R}_6^{(13)} - \bar{\mathsf R}_6^{(2121)}  + \bar{\mathsf R}_6^{(1212)}\big) \in H^{1}(\mathcal P_6^{\; \mathbf{c}\{(121)_\succ\}}; \Z)$,
\item  $\theta_{(3111)} = \mathcal A\ell\big(\bar{\mathsf R}_6^{(411)} - \bar{\mathsf R}_6^{(321)} + \bar{\mathsf R}_6^{(312)}\big) \in H^{2}(\mathcal P_6^{\mathbf{c}\{ (3111)_\succ\}}; \Z)$,
\item   $\theta_{(31)}  = \mathcal A\ell\big(\bar{\mathsf R}_6^{(4)} - \bar{\mathsf R}_6^{(231)} + \bar{\mathsf R}_6^{(321)}  - \bar{\mathsf R}_6^{(312)}\big) \in H^{2}(\mathcal P_6^{\mathbf{c}\{ (31)_\succ\}}; \Z)$,
\item  $\theta_{(1221)} = \mathcal A\ell\big(\bar{\mathsf R}_6^{(321)} - \bar{\mathsf R}_6^{(141)} + \bar{\mathsf R}_6^{(123)}\big)  \in H^{2}(\mathcal P_6^{\mathbf{c}\{(1221)_\succ\}}; \Z)$. \hfill $\diamondsuit$
\end{itemize}
}
\end{example}

Theorem \ref{th.counting_trajectories} below gives a simple recipe for manufacturing traversing flows  with the desired number of $v$-trajectories of a given combinatorial type $\omega$ on compact smooth manifolds $X$ with boundary. Although the resulting construction is explicit, the topological nature of the pull-back $X$ is opec due to, paraphrasing Thom \cite{Th}, \cite{Th1}, ``the misterious nature of transversality".  \smallskip


\begin{theorem}\label{th.counting_trajectories} For any $\om \in \mathbf\Om_{\langle d]}$, let  $\theta_{\om} \in H^{|\om|'}(\mathcal P_d^{\mathbf{c}\{\om_\succ\}}; \Z)$ be the cocycle that takes the value $\mathsf{lk}(Z,\, \d\bar{\mathsf R}^{\om})$ on each $|\om|'$-dimensional cycle $Z$ from   $H_{|\om|'}(\mathcal P_d^{\mathbf{c}\{\omega_\succ\}}; \Z)$. 
Let $Y$ be an oriented closed and smooth $|\omega|'$-dimensional manifold. We denote by $\hat v$ a non-vanishing vector field that is tangent to the fibers of the projection $\R \times Y \to Y$.  Let a map  $\Phi: Y \to  \mathcal P_{d}^{\; \mathbf{c}(\omega_\succ)}$  be $(\d\mathcal E_d)$-regular in the sense of Definition 3.4 from \cite{K9} (see also (\ref{eq.E}))\footnote{By \cite{K9}, Corollary 3.2, such maps form and open and dense set in the space of all smooth maps.}. 
\begin{itemize}
\item Then the natural coupling  $$\langle \Phi^\ast(\theta_\om),\, [Y] \rangle = \langle \Phi^\ast(\mathcal A\ell(\d\bar{\mathsf R}_d^\om)),\, [Y] \rangle$$ gives the \emph{oriented} count of the $\hat v$-trajectories of the combinatorial type $\om$ in the $(|\om|' +1)$-dimensional convex envelop $(\R \times Y,\,\hat v\big)$ of the pair $$\big(X_\Phi =_{\mathsf{def}}  \{(u, y)|\;  \Phi(y)(u) \leq 0\},\, \hat v\big)  \subset \big(\R \times Y,\,\hat v\big).$$ The combinatorial types of the $\hat v$-trajectories relative to $\d X_\Phi$ belong to the poset $\mathbf{c}\{\om_\succ\} = \om_\preceq =_{\mathsf{def}} \{\om' \in \mathbf\Om_{\langle d]} | \;\om' \succeq \om\}$.

\item If $\Phi$ is transversal to the cell $(\mathsf R_d^{\om})^\circ$ in $\mathcal P_{d}^{\; \mathbf{c}\{\om_\succ\}}$, then the  number of $\hat v$-trajectories in $\R \times Y$, whose intersection with $\d X_\Phi$  has the combinatorial type $\om$, equals the cardinality  of the intersection $\Phi(Y) \cap (\mathsf R_d^{\om})^\circ$. Thus the number of such trajectories is greater than or equal to the absolute value   $|\langle \Phi^\ast(\theta_\om), [Y] \rangle|$.
\end{itemize}
\end{theorem}

\begin{proof}  The argument is similar to the proof of Lemma \ref{lem.char_classes}.
By  Theorem \ref{thA}, the homology class of the cycle $\d\bar{\mathsf R}_d^{\omega}$ in $H_{d- |\omega|' -1}(\bar{\mathcal P}_d^{\{\om_\succ\}}; \Z)$, represented by formula (\ref{eq.d+d}), is nontrivial since $d- |\omega|' -1$ is the top grading of the differential complex $(\Z[\omega_\succ], \d)$.

The proof amounts to spelling out the nature of Alexander duality $\mathcal A\ell$. By its definition, $\langle Z,\, \mathcal A\ell(\d\bar{\mathsf R}_d^{\omega}) \rangle = \mathsf{lk}(Z, \d\bar{\mathsf R}_d^{\om})$ for any cycle $Z$ in  $\mathcal P_{d}^{\; \mathbf{c}(\om_\succ)}$ of dimension $|\om|'$. Thus
$$\langle \Phi^\ast(\theta_\om), [Y]  \rangle = \langle \theta_\om, \Phi_\ast([Y])  \rangle = \mathsf{lk}(\d\bar{\mathsf R}_d^{\om},\,  \Phi(Y)) = \bar{\mathsf R}^{\om}_d\, \circ \, \Phi(Y).$$

Examining the construction of the space $X_\Phi \subset Y \times \R$, given by $(\mathsf{id} \times \Phi)^{-1}(\mathcal E_d)$, where the hypersurface $\mathcal E_d$ was introduced in (\ref{eq.E}), the latter intersection number gives an \emph{oriented} count of the $\hat v$-trajectories in $\R \times Y $ of the combinatorial type $\om$ relative to the boundary $\d X_\Phi$.

If $\Phi: Y \to  \mathcal P_{d}^{\; \mathbf{c}(\om_\succ)}$ is transversal to the open cell $(\mathsf R^{\om}_d)^\circ$, then the cardinality of the geometric intersection $\Phi(Y) \cap (\mathsf R^{\om})^\circ$ is exactly the total number of $\hat v$-trajectories of the combinatorial type $\om$ with respect to $\d X_\Phi$. The intersection points from $\Phi(Y) \cap (\mathsf R_d^{\om})^\circ$ (equivalently, the trajectories of the type $\om$ with respect to $\d X_\Phi$) come in two flavors, $\{\oplus, \ominus\}$, depending on whether the canonicalal normal orientation of the cell $(\mathsf R_d^{\om})^\circ$ in $\mathcal P_d$ agrees or disagrees with the preferred orientation of the cycle $\Phi(Y)$.
\end{proof}


The next corollary follows directly from Theorem \ref{th.counting_trajectories}.

\begin{corollary}\label{cor.(1221)} Take $\om = (1, \underbrace{2, \dots , 2}_{n}, 1)$.  
Let  $\Phi: Y^n \to \mathcal P_{2n +2}^{\; \mathbf{c}\{\om_\succ\}}$ be as in Theorem \ref{th.counting_trajectories}.

Then  the \emph{oriented}  number of  trajectories of the combinatorial type $\om$ in the pull-back $X^{n+1}_\Phi \subset \R \times Y^n $ equals the linking number of the cycle $\Phi(Y^n)$ with the $(n +1)$-cycle
$$\d\bar{\mathsf R}_{2n +2}^{(12\dots 21)} = \bar{\mathsf R}_{2n +2}^{(32\dots 21)} - \bar{\mathsf R}_{2n +2}^{(142 \dots 21)} + \dots + (-1)^{n -1}\,\bar{\mathsf R}_{2n +2}^{(12 \dots 241)} + (-1)^n\,\bar{\mathsf R}_{2n +2}^{(12 \dots 23)}.$$

In particular, the number of $\om$-tangent trajectories in $X^{n+1}_\Phi$ is at least $$|\, \mathsf{lk}(\Phi(Y^n),\, \d\bar{\mathsf R}_{2n +2}^{(12\dots 21)})|= |\langle h^\ast(\theta_{(12\dots 21)}),\, [Y] \rangle |. \quad \quad \hfill \diamondsuit$$  
\end{corollary}

\subsection{Comments and questions} The following basic, however, non-trivial question animates many of our previous investigations: 

\emph{``For a given closed poset $\Theta \subset \mathbf\Om_{\langle d]}$, what are the restrictions on the topology of compact manifolds $X$ that admit traversing $v$-flows (or their pseudo-envelops/envelops $(\hat X, \hat v)$) whose tangency patterns avoid $\Theta$?"} \smallskip

In particular, what smooth topological types $X$ may arise via classifying maps $\Phi$ (say, as in Theorem \ref{th.counting_trajectories} or Corollary \ref{cor.(1221)})? We do not have any problems with manufacturing $\d\mathcal E_d$-regular maps from a given compact $n$-manifold $Y$ to $\cP_d^{\mathbf c\Theta}$, $d \equiv 0 \mod 2$, and using such maps to produce embeddings $\a: X \subset \R \times Y$ and traversing flows on $X$ that avoid the $\Theta$-patterns. However, the topological types of such $X$'s are beyond our control: we get what we get... The resulting $X$ is subject to many restrictions, whose nature we do not understand conceptually. Let us sketch just a couple examples which indicate that the restrictions on $X$ by $\Theta$ can be severe. 

In \cite{K10}, and \cite{K7}, Chapters 2-4, we proposed an answer this question for $3$-folds $X$. For them, the \emph{minimal} number $gc(X)$ of $v$-trajectories of the combinatorial type $\om = (1221)$ was introduced as \emph{a measure of complexity} of a $3$-fold $X$ and was linked directly to the combinatorial complexity $c(X)$ of their $2$-spines. Via this link, $gc(X)$ is related to the classification of $3$-folds. For instance, a connected compact oriented $3$-fold with a simply-connected boundary, which admits a traversing flow that avoids $\Theta = (1221)$, is a connected sum of several $3$-balls and products $S^1 \times S^2$ (\cite{K10}, Theorem 3.14). Therefore, no other $3$-folds with a simply-connected boundary can admit a convex pseudo-envelop or even a traversing flow that avoids the tangency pattern $(1221)$.

In a quite different setting, consider a closed hyperbolic $(n+1)$-manifold $Z$. Let $X$ be obtained from $Z$ by deleting a number $(n+1)$-balls. Then no such $X$ can support a traversing $v$-flow that avoids the set of isolated $\hat v$-trajectories of combinatorial types $\{\om \in \mathbf \Om_{|\sim|' = n}\}$. Indeed, the positivity of Gromov's simplicial semi-norm $\|Z\|_\D$ rules out the possibility of such an avoidance (see \cite{AK}, Theorem 1). In fact, the positivity of simplicial semi-norms of various homology classes $h \in H_\ast(X; \R)$ is the only general mechanism known to us that imposes constraints on  the combinatorial tangency types of any generic traversing flow on $X$ \cite{K11}. 
\smallskip

Let us conclude with the following remark. Unfortunately, we do not know examples of invariants that distinguish between the quasitopies  $\mathcal{QT}^{\mathsf{emb}}_{d, d'}(\hat X,\hat v; \mathbf{c}\Theta; \mathbf{c}\Theta')$ and $\mathcal{QT}^{\mathsf{sub}}_{d, d'}(\hat X, \hat v; \hfill \break \mathbf{c}\Theta; \mathbf{c}\Theta')$ of convex envelops and pseudo-envelops. It seems natural to think about invariants that utilize the $k$-multiple self-intersection manifolds  $\{\Sigma_k^{\a^\d}\}_{k \geq 2}$ of $\a^\d$ (see Remark \ref{rem.Sigma}). However, for the convex pseudo-envelops, the analogue of the distinguishing map (3.3) in \cite{K9} is trivial: in fact, for a submersion $\a: X \to \hat X$, the bordism class of the map $\pi \circ \a^\d \circ p_1: \Sigma_k^{\a^\d} \to \hat X$ vanishes, due to arguments as in \cite{K9}, Corollary 3.1. As a result, a direct analogue of Proposition 3.6  from \cite{K9} is vacuous in the environment of convex envelops. 
\smallskip


{\it Acknowledgment:} The author is grateful to the Department of Mathematics of Massachusetts  Institute of Technology for many years of hospitality.


\end{document}